\documentclass[11pt, a4paper]{amsart}
\usepackage[utf8]{inputenc}
\usepackage[T1]{fontenc}
\usepackage{comment}
\usepackage{enumerate}
\usepackage[english]{babel}
\usepackage{amsmath}
\usepackage{amsthm}
\usepackage{amsfonts}
\usepackage{amssymb}
\usepackage[foot]{amsaddr}
\usepackage{graphicx}
\usepackage[usenames,dvipsnames]{color}
\usepackage[normalem]{ulem}
\usepackage{geometry}
\usepackage{pgf,tikz,pgfplots}
\pgfplotsset{compat = 1.16}
\usepackage[shortlabels]{enumitem}

\usepackage[
backend=biber,
natbib=true,
style=numeric,
isbn=false,
maxbibnames=9,
url=false,
doi=false
]{biblatex}
\usepackage{csquotes}
\renewbibmacro{in:}{}
\DeclareFieldFormat{pages}{#1}

\DeclareFieldFormat{title}{\mkbibquote{#1}}
\DeclareFieldFormat[misc]{date}{Preprint (#1)}

\newbibmacro{string+doi}[1]{%
  \iffieldundef{doi}{\iffieldundef{url}{#1}{\href{\thefield{url}}{#1}}}{\href{http://dx.doi.org/\thefield{doi}}{#1}}}	
\DeclareFieldFormat*{title}{\usebibmacro{string+doi}{``#1''}}

\usepackage{hyperref}
\hypersetup{colorlinks, linkcolor={red!50!black}, citecolor={green!50!black}, urlcolor={blue!50!black}}

\usepackage{xurl}
\hypersetup{breaklinks=true}

\theoremstyle{plain}
\newtheorem{theorem}{Theorem}[section]
\newtheorem{corollary}[theorem]{Corollary}
\newtheorem{proposition}[theorem]{Proposition}
\newtheorem{lemma}[theorem]{Lemma}

\newtheorem{claim}[theorem]{Claim}

\theoremstyle{definition}
\newtheorem{remark}[theorem]{Remark}
\newtheorem{definition}[theorem]{Definition}

\newtheorem{problem}[theorem]{Open problem}

\newcommand{\nicoc}[1]{{\color{Red}{ \bf [~Nicolas:\ }\emph{#1}\textbf{~]}}}

\renewcommand{\sp}{\mathrm{sp}}

\newcommand{\sh}{\mathrm{sh}}
\newcommand{\e}{\mathrm{e}}
\newcommand{\eps}{\varepsilon}
\newcommand{\Co}{\mathrm{Co}_2}

\usepackage{dsfont}

\newcommand{\cA}{\mathcal {A}}
\newcommand{\cB}{\mathcal {B}}
\newcommand{\cC}{\mathcal {C}}
\newcommand{\cD}{\mathcal {D}}
\newcommand{\cE}{\mathcal {E}}
\newcommand{\cF}{\mathcal {F}}
\newcommand{\cG}{\mathcal {G}}
\newcommand{\cH}{\mathcal {H}}

\newcommand{\cL}{\mathcal {L}}
\newcommand{\cM}{\mathcal {M}}

\newcommand{\cP}{\mathcal {P}}

\newcommand{\cV}{\mathcal {V}}

\title{ 
    Vertex-separating path systems in random graphs
}

\author[Lyuben Lichev]{Lyuben Lichev$^1$}
\address{$^1$Institute of Science and Technology Austria (ISTA), 3400 Klosterneuburg, Austria}
\author[Nicolás Sanhueza-Matamala]{Nicol\'as Sanhueza-Matamala$^2$}
\address{$^2$Departamento de Ingenier\'ia Matem\'atica, Facultad de Ciencias F\'isicas y Matem\'aticas, Universidad de Concepci\'on, Chile}

\begin{document}

\begin{abstract}
A set $V$ is said to be \emph{separated} by subsets $V_1,\ldots,V_k$ if, for every pair of distinct elements of $V$, there is a set $V_i$ that contains exactly one of them. 
Imposing structural constraints on the separating subsets is often necessary for practical purposes and leads to a number of fascinating (and, in some cases, already classical) graph-theoretic problems.

In this work, we are interested in separating the vertices of a random graph by path-connected vertex sets $V_1,\ldots,V_k$, jointly forming a \emph{separating system}. 
First, we determine the size of the smallest separating system of $G(n,p)$ when $np\to \infty$ up to lower order terms, and exhibit a threshold phenomenon around the sharp threshold for connectivity. 
Second, we show that random regular graphs of sufficiently high degree can typically be optimally separated by $\lceil \log_2 n\rceil$ sets.
Moreover, we provide bounds for the minimum degree threshold for optimal separation of general graphs.
\end{abstract}

\maketitle

\section{Introduction}

Encoding objects via unique identifiers that could be efficiently kept, modified and consulted is an idea perhaps as old as the notion of counting itself.
After positive integers, binary vectors are maybe the most standard and natural choice of identifiers.
Given a set $\Pi$ of objects encoded by binary vectors of length $\ell$ and a binary vector $x = (x_1,\ldots,x_{\ell})$, 
one way to find the object with identifier $x$ (or confirm that no such object exists) is to separate $\Pi$ according to the first coordinate of the identifiers, 
then separate the subset of elements of $\Pi$ with first coordinate $x_1$ according to the second coordinate, and so on.

More generally, given a set $\Pi$ of size $n$, we say that a collection of sets $\Pi_1, \ldots, \Pi_{\ell} \subseteq \Pi$ is a \emph{separating system} of $\Pi$ if, for every two distinct elements $u,v\in \Pi$, there is $i\in [\ell]$ such that $u\in \Pi_i, v\notin \Pi_i$ or vice versa.
This concept was introduced by R\'enyi~\cite{Renyi1961} who constructed some special separating systems of optimal size $\ell = \lceil \log_2 n\rceil$ to separate the elements of finite Boolean algebras.
Finding separating systems of that size without further restrictions is an immediate task, but it is often useful to have an additional property satisfied by each of the sets.
Understanding the minimum size of a separating system given some additional structural restrictions on the set $\Pi$ is a topic that has drawn considerable attention in several different settings. 

\subsection{Previous work on separating systems}
 
In one particular setting, the set $\Pi$ to be separated is a set of elements of a graph $G$ (its edges or its vertices) and the sets in the separating system we are looking for correspond to subgraphs of $G$ with a particular structure (e.g.\ we might only wish to use subsets forming a path in $G$).
Naturally, the interest is to separate $\Pi$ efficiently, that is, with the smallest possible number of sets.

Separation in graphs using paths is motivated by the application of efficiently detecting faulty links or nodes in networks and was considered many times by computer scientists~\cite{FK12,HPWYC2007,TapolcaiRonyaiHo2013,ZakrevskiKarpovsky1998}.
Recently, there has been a lot of exciting research on the \emph{edge-separation} variant of the problem.
A famous conjecture in the area (raised independently by several researchers, see~\cite{BCMP2016,FKKLN2014}) said that $O(n)$ paths should be able to separate $E(G)$ in any $n$-vertex graph $G$.
Following a major progress of Letzter~\cite{Let24} who proved that the edges of any $n$-vertex graph $G$ can be separated with $O(n \log^* n)$ paths,\footnote{Here, $\log^* n$ is the iterated
logarithm defined as the minimum number of times the base-2 logarithm has to be applied to yield a result smaller than 1.}
the conjecture was settled by Bonamy,
Botler, Dross, Naia, and Skokan~\cite{BBDNS2023}.
In the special case of the complete graph on $n$ vertices, Kontogeorgiou and Stein~\cite{KS24} recently showed that the smallest separating family of paths is known to have size between $n-1$ and $n+2$, see also~\cite{FGS-M24,Wic24}.

In this paper, we focus on the variant of the problem considering \emph{vertex-separation using paths}.
Given a graph $G$, we define $\sp(G)$ as the minimum size of a family of paths $\mathcal{P}$ such that $\{V(P) : P \in \mathcal{P} \}$ is a separating system of $V(G)$.
Analysing $\sp(G)$ for different graph families is a natural alternative to the edge-separating problem and has been studied by a number of authors.
Foucaud and Kov\v{s}e~\cite{FoucaudKovse2013} determined the size of minimum vertex-separating path systems whose paths cover $V(G)$ whenever $G$ is a path, a cycle or a hypercube, provided upper bounds for trees and studied the complexity of computing $\sp(G)$.
Two more recent works~\cite{AAACGHOS-MT-C23, BBdCMMOSSY2023} provided a more complete analysis of $\sp(G)$ on trees. 
Honkala, Karpovsky, and Litsyn~\cite{HonkalaKarpovskyLitsyn2003} and Rosendahl~\cite{Rosendahl2003} investigated the related notion of vertex-separating cycle systems 
in hypercubes, complete bipartite graphs and grids.

The fact that the number of vertex-separating path systems of $K_n$ of trivial size $\lceil \log_2 n \rceil$ is enormous suggests that one should be able to separate a lot sparser graphs on $n$ vertices with the same number of paths.
We focus on understanding if there is a sharp threshold for this property and, if so, where this threshold is and how $\sp(G(n,p))$ behaves for smaller values of $p$.
This question was partially answered by Arrepol et al.~\cite[Theorem 3.5]{AAACGHOS-MT-C23}, where it was shown that if $np - 2 \log n = \omega(\log \log n)$, then typically $\sp(G(n,p)) \leq \lceil \log_2 n \rceil + 1$, and if $\log n - np = \omega(\log \log n)$, then typically $\omega(\log n)$ paths are required.

\subsection{Our contributions}

\subsubsection{Vertex-separation of binomial random graphs}

First of all, we vastly generalise~\cite[Theorem 3.5]{AAACGHOS-MT-C23} by conducting a thorough analysis of the minimum size of a vertex-separating path system of the binomial random graph. 
We recall that a sequence of events $(A_n)_{n\geq 1}$ is \emph{asymptotically almost sure} (abbreviated \emph{a.a.s.}) if $\mathbb P(A_n)$ tends to 1 as $n$ goes to infinity.
For three numbers $a,b,c$ with $c > 0$, we use the classical notation $a = b\pm c$ to say that $a\in [b-c,b+c]$.

\begin{theorem}\label{thm:main_ER}
Fix a sufficiently small $\eps > 0$, $p = p(n)\in [0,1]$ and $G\sim G(n,p)$.
\begin{enumerate}[\upshape{(\roman*)}]
    \item \label{item:main-dense} If $p\ge (1+\eps)\log n/n$, then a.a.s.\ $\sp(G) = \lceil \log_2 n\rceil$.
    \item \label{item:main-critical} If $p = (1\pm\eps)\log n/n$, then a.a.s.\
    \begin{equation}\label{eq:(ii)}
    \sp(G) = (1+o(1))\max\left(\log_2 n, \frac{2n^2p\e^{-np}}{3}\right).    
    \end{equation}
    \item \label{item:main-sparse} For every $\delta > 0$, there is a constant $C = C(\delta)$ such that, if 
    \[p\in [C/n, (1-\eps)\log n/n],\] 
    then a.a.s.\ $\sp(G) = (2/3\pm \delta) n^2p \e^{-np}$.
\end{enumerate}
\end{theorem}

Thus, Theorem~\ref{thm:main_ER} precisely describes the evolution of $\sp(G(n,p))$ for all $p$ sufficiently far from the sharp threshold for existence of a giant component.
We remark that our proofs of Parts~\ref{item:main-critical} and~\ref{item:main-sparse} actually yield quantitative bounds for the error terms that are spared here for simplicity of the exposition.

Our inability to derive a precise expression for $\sp(G(n,p))$ when $np = \Theta(1)$ stems from the fact that a constant proportion of the vertices of $G$ belong to small connected components without cycles,
and a precise description of $\sp(T)$ is not available for general trees $T$.
Nevertheless, when $\lambda\in (0,1]$ and $p = \lambda/n$, the convergence in probability of $\sp(G(n,p))/n$ to a constant follows from the fact that the number of connected components isomorphic to any fixed tree is well concentrated around its expected value and only $o(n)$ vertices belong to components containing a cycle.
Showing concentration for the size of the smallest vertex-separating path systems of the giant component when $\lambda > 1$ is more difficult since we cannot fully characterise $\sp(\cdot)$ in terms of the (well understood) structure of the giant.
A logical step in this case would be to turn to non-constructive methods like martingale concentration techniques.
Our next proposition shows that the above approach is not immediate as the addition of one edge can cause a linear jump of $\sp(\cdot)$ even in sparse graphs.

\begin{proposition}\label{prop:example}
For every $n\ge 60$, there is a graph $G$ with $n$ vertices, $2n-8$ edges and an edge $e$ outside $G$ such that $\sp(G) - \sp(G\cup \{e\}) \ge n/6-10$.
\end{proposition}

Our construction relies on the fact that the graph $G$ in Proposition~\ref{prop:example} contains independent sets of linear size that belong to the same connected component but only have a constant number of neighbours. 
Such subgraphs are highly unlikely to be found in the binomial random graph, so we still believe that $\sp(G(n,\lambda/n))/n$ should converge in probability to a constant for every $\lambda > 0$.

\begin{problem}
Show that, for every $\lambda > 1$ and $p = p(n) = \lambda/n$, there is a constant $c = c(\lambda)\in (0,1]$ such that a.a.s.\ $\sp(G(n,p)) = cn + o(n)$.
\end{problem}

At the same time, the structure of the random graph suggests that an explicit characterisation of the constant $c(\lambda)$ in the above problem in terms of $\lambda$ goes through understanding $\sp(T)$ for general trees $T$.

\subsubsection{Vertex-separation of random regular graphs}

The reader may have recognised that the intermediate regime in Theorem~\ref{thm:main_ER} consists of a small window around the sharp threshold for connectivity for the binomial random graph.
On the one hand, $\sp(G(n,p)) = \sp(K_n)$ when $p\ge (1+\eps)\log n/n$ since $G(n,p)$ is typically highly connected in this regime.
On the other hand, when $p$ is below $\log n/n$, the large number of vertices of degree 0 or 1 in the graph directs the behaviour of $\sp(G(n,p))$: indeed, isolated vertices cannot participate in non-trivial paths while vertices of degree 1 can only serve as their endpoints.

It is thus a natural question if sparser graphs with good connectivity properties behave like the complete graph from point of view of the smallest vertex-separating path systems or there is another significant obstruction not allowing such systems of trivial size.
Here, we concentrate on answering this question for random regular graphs.
Given a positive integer $d\in [n-1]$, the \emph{random $d$-regular graph} $G(n,d)$ is distributed uniformly over the set of $d$-regular graphs on $n$ vertices.
(Note that $nd$ has to be even to avoid triviality of the definition. In the sequel, we work under this assumption, often without further mention.)

It turns out that one can derive that a.a.s.\ $\sp(G(n,d)) = \lceil \log_2 n\rceil$ when $d = \omega(\log n)$ from Theorem~\ref{thm:main_ER}(i) using coupling arguments based on `sandwiching' random regular graphs with Erd\H{o}s--Rényi random graphs~\cite{Gao23,KV04}; we give the details of this argument in Section~\ref{section:hamsandwich}.
However, we can go further and manage to show that a.a.s.\ $\sp(G(n,d)) = \lceil \log_2 n\rceil$ even when $d$ is a sufficiently large constant. 

\begin{theorem}\phantomsection\label{thm:main_reg}
There is an integer $D\ge 1$ such that, for every integer $d\ge D$ and $n\to \infty$ such that $dn$ is even, a.a.s.\ 
\[\sp(G(n,d)) = \lceil \log_2 n\rceil.\]
\end{theorem}

\noindent
We also believe that the conclusion of Theorem~\ref{thm:main_reg} could hold for smaller values of $d$ as well.

\begin{problem}
Prove or disprove that, for all $d\ge 3$ and $n\to \infty$ such that $dn$ is even, a.a.s.\ $\sp(G(n,d)) = \lceil \log_2 n\rceil$.
\end{problem}

\subsubsection{Extremal conditions for vertex-separation of general graphs}

Our last result says more about the minimum degree threshold for optimal vertex-separation.
More precisely, let $f(n)$ be the minimum value of $t$ such that each $n$-vertex graph $G$ with minimum degree at least $t$ satisfies $\sp(G) = \sp(K_n) = \lceil \log_2 n \rceil$.
We determine $f(n)$ up to lower order terms.

\begin{proposition}\label{prop:f(n)}
For every $n\ge 1$, 
\[\left\lceil \frac{n-1}{2}\right\rceil \leq f(n) \leq \frac{n}{2} + 9\sqrt{n \log\log n}. \]
\end{proposition}

We believe that the lower bound is closer to the truth and the minimum degree threshold for optimal vertex-separation essentially coincides with the threshold for existence of a Hamilton path.

\begin{problem}
Is $f(n) - \lceil (n-1)/2\rceil$ bounded from above by a uniform constant? 
\end{problem}

\subsection{Outline of the proofs} 
Our proofs use various techniques. Some of the results claimed below only hold asymptotically almost surely but we avoid specifying this here for the sake of a simplified exposition.

\subsubsection{Separating with $\lceil \log_2 n \rceil$ paths}
To begin with, we prove strengthened versions of Theorems~\ref{thm:main_ER}(i) and~\ref{thm:main_reg} where we separate using cycles instead of paths (see Propositions~\ref{prop:strengthen_thm1(i)} and~\ref{prop:strengthen_thm2}, respectively).

The proof of Proposition~\ref{prop:strengthen_thm1(i)} is divided into two parts. 
First, we show (in Lemma~\ref{lemma:tight}) that, if a graph satisfies a certain minimum degree criterion for Hamiltonicity of induced subgraphs with $\lceil n/2\rceil$ vertices, then it can be separated by $\lceil \log_2 n\rceil$ vertex sets inducing Hamiltonian subgraphs of $G(n,p)$. The proof is based on an iterative application of the Lov\'asz Local Lemma.
Secondly, we use a result from Araujo, Pavez-Signé, and the second author~\cite{APS2022} certifying that $G(n,p)$ satisfies this criterion away from the connectivity threshold.

In turn, the proof of Theorem~\ref{thm:main_reg} combines Lemma~\ref{lemma:tight}, a recent breakthrough showing that all sufficiently good expanders are Hamiltonian~\cite{DMCPS24} and results from~\cite{BFSU98,Pav23} justifying that all subgraphs of $G(n,d)$ with sufficiently high minimum degree have good expansion properties.

\subsubsection{Separating systems at the critical threshold}

Our proof of Theorem~\ref{thm:main_ER}(ii) is more involved and contains two main parts.

First, we show a slightly weaker version of the result where the lower bound matches the right hand side of~\eqref{eq:(ii)} while the upper bound is given by the sum of the two terms in the maximum (see Proposition~\ref{prop:critical1}). 
To this end, it turns out that the leaves and the isolated vertices can be separated easily by the asymptotically optimal number of $(2/3+o(1))n^2p\e^{-np}$ paths, and the difficulty stems from the necessity to separate the 2-core of the random graph with $(1+o(1))\log_2 n$ paths.
To achieve this, we design a procedure that randomly attributes binary vectors with several key properties to the vertices of the graph. More precisely, we need that these vectors have length $2\ell = \log_2 n + O(\log\log n)$, weight $\ell$ (that is, containing exactly $\ell$ 1-bits) and the pairwise Hamming distances between them are bounded from below by a suitably large constant $C_1$ (ensuring good error-correction capability of the introduced binary code).
Then, we consider the sets $S_1, \ldots, S_{2\ell}$ consisting of all vertices whose corresponding vector has 1 in coordinate $1,\ldots,2\ell$, respectively.
By treating high-degree ($\ge \eps \log n$) and low-degree ($< \eps \log n$) vertices separately, we manage to show that the 2-cores of $G[S_1], \ldots, G[S_{2\ell}]$ (where $G\sim G(n,p)$) are Hamiltonian and successfully separate the $2$-core of $G(n,p)$; indeed,
\begin{itemize}
    \item on the one hand, for every high-degree vertex $v$ and all but at most $C_1/2$ of sets $S_i$ containing $v$, we show that $v$ belongs to the 2-core of $G[S_i]$. Since our vectors have weight $\ell$ and the Hamming distance between every pair is at least $C_1$, all high-degree vertices are separated from all other vertices.
    \item On the other hand, to ensure that the 2-cores also separate the low-degree vertices between themselves, we show that there are only few such vertices and their corresponding vectors are typically at Hamming distance at least $\ell/2$ from each other. This allows us to conclude by a union bound over the bad events that a given pair fails to be separated.
\end{itemize}

While Proposition~\ref{prop:critical1} provides an asymptotically optimal result when $|np-\log n|\gg 1$ up to lower order terms, in the case when the number of leaves is comparable to $\log_2 n$, we manage to go further in our analysis and replace the sum in the upper bound with a maximum (see Proposition~\ref{prop:maxcase}). 
    To this end, our goal is to show that some paths can be used to simultaneously separate the leaves and the vertices in the 2-core of $G$. This is intrinsically related to understanding the existence of Hamilton paths with prescribed endpoints $x,y$ in the 2-core of random graphs around $np = \log n/2$.
    While much is known about the existence of Hamilton cycles in 2-cores of sufficiently dense random graphs (established by {\L}uczak in 1987~\cite{Luc87}), the fixed endpoints requirement makes the problem significantly less tractable. 
    Our strategy to overcome this difficulty is to define an auxiliary graph $H^*$ with the property that any Hamilton cycle in it can be used to construct a path from $x$ to $y$ spanning the 2-core of the original graph (see Definition~\ref{def:H*} and Lemma~\ref{lemma:hstarhamilton}).
    The definition of the auxiliary graph has two steps. 
    \begin{itemize}
        \item First, we introduce a new vertex $z$, connect it by an edge to one neighbour of each of the prescribed endpoints $x,y$, and finally delete $x$ and $y$.
        \item Second, iteratively and as long as possible, we find a short path $P$ between low-degree vertices $u,v$ in the 2-core, delete $P$ and connect a neighbour of $u$ and a neighbour of $v$ outside $P$ by a path of length 2.
    \end{itemize}
Next, we prove that $H^*$ has good expansion properties. This is done by treating the set $U^*$ of surviving original high-degree vertices, the set $S^*$ of surviving original low-degree vertices and the set $T^*$ of newly introduced vertices separately and using that the vertices in $S^*\cup T^*$ are pairwise at graph distance at least 5 in $H^*$.
Finally, we adapt an argument of Lee and Sudakov~\cite{LS12} reducing the problem of Hamiltonicity of $H^*$ to finding sufficiently sparse spanning expanders $F^*\subseteq H^*$. 
The argument relies on iterative applications of the fact that, for any such $F^*$, there must typically be many edges in $G[S^*\cup U^*]\setminus F^*$ that could serve to extend the longest path in $F^*$ or complete a Hamilton cycle, unless one already exists.
The proof is completed by a randomised construction of a sparse expander $F^*$ with the said properties.

\subsubsection{Separating systems below the critical threshold}
The proof of Theorem~\ref{thm:main_ER}(iii) uses similar ideas as the previous proof. First, we sparsify the 2-core of the giant component of $G$ while preserving several structural properties (among them, the fact that large sets have good expansion).
Then, we separate low-degree vertices between themselves and from the rest; again, these are divided into two groups: ones which have another low-degree vertex at distance at most 4 and the rest.
Finally, we use $(\log n)^4$ sets to separate the vertices of high degree in the 2-core.
Here, note that the number of leaves in the regime of interest is typically polynomial in $n$, so $(\log n)^4$ is a lower order term in the total count.

\subsubsection{Separating systems in graphs with large minimum degree}
The paper ends with a proof of Proposition~\ref{prop:example} where a short explicit construction is provided, and a proof of Proposition~\ref{prop:f(n)}. 
For the latter, the lower bound comes from the minimum degree threshold for connectivity.
Concerning the upper bound, we first attribute randomly binary vectors of length $\ell = \lceil \log_2 n\rceil$ to the vertices of the graph, define the sets $S_1,\ldots, S_\ell$ similarly to the proof of Theorem~\ref{thm:main_ER}(ii), and finally show that each of the graphs $G[S_1], \ldots, G[S_{\ell}]$ satisfies P\'osa's criterion for Hamiltonicity.

\subsection{Plan of the paper} 
In Section~\ref{sec:prelims}, we include several preliminary results and introduce classical notation and terminology.
In Section~\ref{sec:3}, we show Theorem~\ref{thm:main_ER}(i) and Theorem~\ref{thm:main_reg}. 
In Section~\ref{sec:4}, we show Theorem~\ref{thm:main_ER}(ii) assuming a key result 
whose proof is outsourced to Section~\ref{sec:5}.
In Section~\ref{sec:6}, we show Theorem~\ref{thm:main_ER}(iii). 
Finally, we show Propositions~\ref{prop:example} and~\ref{prop:f(n)} in Section~\ref{sec:7}.

\section{Preliminaries}\label{sec:prelims}

First of all, we recall one instance of the well-known Chernoff's inequality.
\begin{lemma}[see Theorem 2.1 in \cite{JLR11}]\label{lem:chernoff}
For every binomial random variable $X$ and $t\ge 0$,
\begin{equation*}
\max(\mathbb{P}(X - \mathbb E[X] \ge t), \mathbb{P}(X - \mathbb E[X] \le -t)) \le \exp \left( - \frac {t^2}{2 (\mathbb E[X] + t/3)} \right).
\end{equation*}
\end{lemma}

In fact, Chernoff's inequality holds in wider generality.
We say that the Binomial random variables $X_1,\ldots,X_n$ are \emph{negatively correlated} if, for every set $I\subseteq [n]$,
\[\mathbb P\bigg(\bigcap_{i\in I} \{X_i = 1\}\bigg)\le \prod_{i\in I} \mathbb P(X_i = 1).\]

\begin{lemma}[see Theorem~3.4 in~\cite{PS97}]\label{lem:PS}
Fix a sequence $X_1,\ldots,X_n$ of negatively correlated random variables with values in $\{0,1\}$, and let $X = X_1+\ldots+X_n$. Then, for every $t\ge 0$,
\[\mathbb P(X - \mathbb E[X]\ge t)\le \exp\left(-\frac{t^2}{2(\mathbb E[X]+t/3)}\right).\]
\end{lemma}
\noindent
Note that the inequality provided in~\cite{PS97} is for the event $\{X > (1+\eps)\mathbb E[X]\}$, which is the strict multiplicative version of the upper tail in Chernoff's inequality. The non-strict additive version of the upper tail provided in Lemma~\ref{lem:PS} is derived along the same lines.

Another classical probabilistic result needed in our considerations is the following symmetric version of the Lov\'asz Local Lemma. 

\begin{lemma}[Corollary~5.1.2 in~\cite{AS16}]\label{lem:lovasz}
Let $\cE_1, \cE_2, \ldots, \cE_n$ be events in an arbitrary probability space. Suppose that each event $\cE_i$ is mutually independent of a set of all other events $\cE_j$ but at most $d$, and that $\mathbb P(\cE_i)\le p$ for all $i\in [n]$. If $\e p(d+1)\le 1$, then $\mathbb P(\bigcap_{i=1}^n \overline{\cE_i}) > 0$.
\end{lemma}

Next, we state and prove a general lower bound on the minimal size of a vertex-separating path system in terms of the number of leaves in a graph.
While it is implicitly implied by the proofs of several results from~\cite{AAACGHOS-MT-C23} (where a stronger inequality in the particular case of trees is shown) and a version for covering path systems appears as Proposition~12 in~\cite{FoucaudKovse2013}, we provide a proof for completeness.

\begin{lemma}\label{lem:X1}
For every graph $G$ with $\ell$ vertices of degree $1$, $\sp(G)\ge \lceil 2(\ell-1)/3\rceil$.
\end{lemma}
\begin{proof}
Fix a family $\cP$ of $\sp(G)$ paths in $G$ that separates the vertices of $G$.
Denote by $L$ the set of vertices of degree 1 in $G$ and set $\cL = \{P\cap L: P\in \cP\}$. 
Then, $\cL$ is a family of sets of size at most 2 that separates $L$.
Denote by $S$ the set of vertices $u\in L$ such that $\{u\}\in \cL$ and define an auxiliary graph $\Gamma$ with vertex set $L\setminus S$ and edges $uv$ where $\{u,v\}\in \cL$.
Then, by assumption, $\Gamma$ contains at most one isolated vertex and no isolated edges.
Setting $s = |S|$, we get that $\Gamma$ contains at most $\lfloor (\ell-s-1)/3\rfloor$ connected components, each containing at least three vertices. As a result, there are at least $\ell - s - 1 - \lfloor (\ell-s - 1)/3\rfloor = \lceil 2(\ell-s - 1)/3\rceil$ edges.
Hence, 
\[\sp(G) = |\cP|\ge s + \left\lceil \frac{2(\ell-s-1)}{3}\right\rceil\ge \left\lceil \frac{2(\ell-1)}{3}\right\rceil,\]
as desired.
\end{proof}

Finally, we often use the following simple lemma, which is a direct consequence of Chernoff's inequality and a union bound.
\begin{lemma}\label{lem:max_deg}
Fix an integer $n\ge 1$ and $p = p(n)\in [0,1]$ such that $np\le 2\log n$. Then, for every $c\ge 1$, $\Delta(G(n,p)) \le (c+2)\log n$ with probability at least $1 - n^{1-c/8}$.
\end{lemma}

\subsection{Notation and terminology}
We introduce some (mostly standard) notation used in the following sections. All asymptotic notation is taken with respect to $n\to \infty$ unless explicitly mentioned otherwise. 
The logarithm base $\e$ (resp. base 2) is denoted by $\log$ (resp. $\log_2$).

For a graph $G$, we denote by $e(G)$, $\delta(G)$ and $\Delta(G)$ its number of edges, minimum and maximum degrees, respectively. 
Also, for a vertex $v$ in $G$, we denote by $\deg_G(v)$ the degree of $v$ in $G$ and, for two sets $U,V\subseteq V(G)$, we denote by $e_G(U,V)$ the number of edges of $G$ with one endvertex in each of $U$ and $V$. Furthermore, for a set $U\subseteq V(G)$, we denote by $N_G(U)$ the set of vertices outside $U$ with at least one neighbour in $U$, and $N_G[U] = N_G(U)\cup U$. 
For simplicity, we omit the subscript in $\deg_G(\cdot)$, $e_G(\cdot,\cdot)$, $N_G(\cdot)$ and $N_G[\cdot]$ when the graph is clear from the context and write $e(u,V)$ when $U = \{u\}$ is reduced to a single vertex.
For a set $U\subseteq V(G)$, we also say that an edge $uv$ of $G$ is in $U$ if $\{u,v\}\subseteq U$.

\section{\texorpdfstring{Separating with $\lceil \log_2 n\rceil$ sets: proof of Theorems~\ref{thm:main_ER}(i) and~\ref{thm:main_reg}}{}}\label{sec:3}

For a graph $G$, define $\sh(G)$ to be the size of the smallest vertex-separating family of sets inducing graphs which all contain a Hamilton cycle.
In this section, we prove three propositions which readily imply Theorems~\ref{thm:main_ER}(i) and~\ref{thm:main_reg}, respectively.
The first proposition deals with `dense' binomial random graphs, namely Theorem~\ref{thm:main_ER}(i).

\begin{proposition}\label{prop:strengthen_thm1(i)}
Fix $\eps > 0$ and $G\sim G(n,(1+\eps)\log n/n)$. Then, a.a.s.\ $\sh(G) = \lceil \log_2 n\rceil$.
\end{proposition}

The next two propositions deal with random regular graphs and together, they immediately yield Theorem~\ref{thm:main_reg}.
Proposition~\ref{prop:strengthen_regulardense} deals with relatively dense random regular graphs, while Proposition~\ref{prop:strengthen_thm2} deals with sparser random regular graphs.

\begin{proposition}\label{prop:strengthen_regulardense}
For every integer $d = d(n)\in [(\log n)^2, n]$ satisfying that $dn$ is even, a.a.s.\ $\sh(G(n,d)) = \lceil \log_2 n\rceil$.
\end{proposition}

\begin{proposition}\label{prop:strengthen_thm2}
There is an integer $D$ such that, for every integer $d = d(n)\in [D, (\log n)^2]$ satisfying that $dn$ is even, a.a.s.\ $\sh(G(n,d)) = \lceil \log_2 n\rceil$.
\end{proposition}

The first and the third proposition follow from a common key technique which we present in the next subsection.
The second proposition follows from a coupling argument which we present separately.

\subsection{Hamiltonicity via subgraphs with high minimum degree}

The proofs of Propositions~\ref{prop:strengthen_thm1(i)} and~\ref{prop:strengthen_thm2} have the following common ingredient. 
Fix positive integers $d,t\ge 1$ satisfying $2t\le d$ and let $G, G'$ be two $n$-vertex graphs such that $G'$ is a spanning subgraph of $G$.
We say that the pair $(G',G)$ is \emph{$(d,t)$-useful} if for every $S\subseteq V(G)$ with $|S| = \lceil n/2 \rceil$ and $\delta(G'[S]) \geq d/2-t$, the induced subgraph $G[S]$ is Hamiltonian.

\begin{lemma} \label{lemma:tight}
Let $(G',G)$ be a $(d, t)$-useful pair of graphs on $n$ vertices with $\delta(G') \geq d$ and $\e(2\Delta(G')^2+1) \leq \exp(t^2/2d)$.
Then, $\sh(G) = \lceil \log_2 n \rceil$.
\end{lemma}

\noindent
Note that, while we will only use Lemma~\ref{lemma:tight} in the particular case when $G = G'$, we prove the more general statement to allow for possible further applications where the minimum degree of induced subgraphs of $G$ is not the correct criterion for Hamiltonicity. 
As a classical example from the literature on randomly augmented graphs, there are graphs $G'$ that do not satisfy the above criterion with $G = G'$ but only need a small number of additional random edges to form a graph $G\supseteq G'$ suitable for the purpose.

\begin{proof}[Proof of Lemma~\ref{lemma:tight}]
Fix $\ell = \lceil \log_2 n \rceil$.
The lower bound on $\sh(G)$ is trivial so we concentrate on proving that $\sh(G) \le \ell$.
For this upper bound, we iteratively construct a sequence of separating cycle systems $\mathcal{C}_0, \mathcal{C}_1, \ldots, \mathcal{C}_{\ell}$ of $G$ and related auxiliary graphs $H_0, H_1, \ldots, H_{\ell}$ on the same vertex set as $G$. 
In our construction, for all integers $i\in \{0, \dotsc, \ell\}$, we will have $|\mathcal C_i| = i$ and two vertices will be connected in $H_i$ if the cycle system $\mathcal C_i$ does not separate them; in fact, $H_i$ will consist of a set of cliques whose sizes differ by at most 1.

To begin with, set $\mathcal{C}_0 = \emptyset$ and let $H_0$ be the complete graph on $V(G)$. We argue by induction. 
Fix an integer $j\in [\ell]$ and suppose that the cycle systems $\mathcal{C}_0, \ldots, \mathcal{C}_{j-1}$ and the auxiliary graphs $H_0, \ldots, H_{j-1}$ with the said properties have been constructed.
Let $C_1, \dotsc, C_k$ be the (vertex sets of the) cliques constituting $H_{j-1}$.
The construction of $\mathcal C_j$ and $H_j$ is based on the following claim.

        \begin{claim}\label{claim:induction}
            There exists $S \subseteq V(G)$ such that
            \begin{enumerate}
                \item[$1.$] $|S| = \lceil n/2 \rceil$;
                \item[$2.$] for each $i \in [k]$, $|S \cap C_i| \in \{ \lfloor |C_i|/2 \rfloor, \lceil |C_i|/2 \rceil \}$; and
                \item[$3.$] for each $v \in S$, $e_{G'}(v, S) \geq d/2 - t$.
        \end{enumerate}
        \end{claim}

\begin{proof}[Proof of Claim~\ref{claim:induction}]
We construct the set $S \subseteq V(G)$ by following a randomised procedure.
For each $i\in [k]$, arbitrarily select a matching $M_i$ of size $\lfloor |C_i|/2 \rfloor$ on the vertex set $C_i$, and let $M$ be a matching of size $\lfloor n/2 \rfloor$ on the vertex set $V(G)$ (but not necessarily contained in $G$) which contains each of $M_1, \ldots, M_k$.
Let $\{x_1 y_1, \dotsc, x_r y_r\}$ be the edges in the matching $M$.
Then, for each $i \in [r]$, let $z_i \in \{x_i, y_i\}$ be chosen uniformly at random and independently for different $i$-s, and let $S = \{z_1, \dotsc, z_r\} \cup (V(G) \setminus V(M))$.
Observe that $|S| = \lceil n/2 \rceil$ and $|S \cap C_i| \in \{ \lfloor |C_i|/2 \rfloor, \lceil |C_i|/2 \rceil \}$ for each $i \in [k]$.

We claim that with non-zero probability $e_{G'}(v, S) \geq d/2 - t$ holds for each $v \in V(G)$.
Define $\mathcal{E}_v$ to be the event that $e_{G'}(v, S) < d/2 - t$, 
and let $m_v$ be the (deterministic) number of edges $xy$ in $M$ such that both $x$ and $y$ are neighbours of $v$ in $G$. 
Then, the number of neighbours of $v$ in $S$ stochastically dominates the sum of $m_v$ and a binomial random variable $X$ with parameters $d-2m_v$ and $1/2$.
Therefore,
\begin{align*}
\mathbb P(\mathcal E_v) 
&\le \mathbb P(m_v+X < d/2-t) = \mathbb P(X - \mathbb E[X] < -t)\\
&\le  \exp\left(-\frac{t^2}{2(d/2-m_v+t/3)}\right)\le \exp\left(-\frac{t^2}{2d}\right).
\end{align*}
Set $p = \exp(-t^2/2d)$. Also, note that two events $\mathcal{E}_u, \mathcal{E}_v$ are not independent only if there exists $i\in [r]$ such that both $u$ and $v$ have neighbours (in $G'$) in a pair $\{x_i, y_i\}$ of $M$.
A fixed vertex $u$ can have neighbours in at most $\Delta = \Delta(G')$ such pairs, these pairs in total contain at most $2 \Delta$ vertices, and each of these vertices can be adjacent to at most $\Delta$ other vertices in $G'$. 
Thus, we deduce that each event $\mathcal{E}_u$ is mutually independent of the family of all other events $(\cE_v)_{v\in G}$ except at most $2 \Delta^2$ of them.
Moreover, by assumption, the inequality $\e(2\Delta^2+1)p \leq 1$ is satisfied, so the Lov\'asz Local Lemma (Lemma~\ref{lem:lovasz}) implies that there exists a choice of $S$ which avoids all events $\mathcal{E}_v$, as desired.
\end{proof}

Let $S$ be one set with the properties described in Claim~\ref{claim:induction}. 
Combining the first and the third of these properties with the $(d,t)$-usefulness of $(G',G)$ shows that $G[S]$ contains a Hamilton cycle $\cC_j$. 
Moreover, the second property implies that $\cC_j$ separates each clique $C_1, \ldots, C_k$ into two cliques, and the sizes of the new cliques differ by at most 1, which finishes the induction step.

Note that the induction procedure implies that for each $j \in [\ell]$, each clique $C$ of $H_{j-1}$ on at least two vertices is separated into two cliques $C', C''$ in $H_j$, where $|C'|-1, |C''|-1 \leq (|C|-1)/2$ holds.
By the choice of $\ell$, this implies that all the cliques in $H_{\ell}$ have one vertex, so $\cC_{\ell}$ separates the vertices of $G$, as desired.
\end{proof}


\subsection{Separating the binomial random graph with Hamilton cycles}\label{sec:ER_Ham_sep}

To show Proposition~\ref{prop:strengthen_thm1(i)}, we combine Lemma~\ref{lemma:tight} with the following result due to Ara\'ujo, Pavez-Sign\'e and the second author~\cite{APS2022}.

\begin{lemma}[see Lemma 5.1 in~\cite{APS2022}]\label{lem:pancyclicity}
For each $\lambda > 0$, there exist $B, C, \eps > 0$ such that the following holds asymptotically almost surely. Let $p \ge C/n$ and $G \sim G(n, p)$. 
Then, for all subgraphs (not
necessarily induced) $H\subseteq G$ such that
\begin{enumerate}[\upshape{(\roman*)}]
    \item $|V(H)| > \lambda n$,
    \item $\delta(H)\ge \lambda np$, and
    \item $e(G[V(H)])-e(H)\le \eps pn^2$;
\end{enumerate}
it holds that $H$ contains a cycle of length $\ell$ for all $\ell\in \{B \log n,\ldots, |V (H)|\}$.
\end{lemma}

\begin{proof}[Proof of Proposition~\ref{prop:strengthen_thm1(i)}]
    Let $G \sim G(n,p)$, and let $\eta > 0$ be such that a.a.s.\ $\eta np\le \delta(G)\le \Delta(G)\le np/\eta$.
    Let $d=\eta np$ and $t=\eta np/4$.
    It suffices to show $(G,G)$ a.a.s. forms a $(d,t)$-useful pair, as then we are done by Lemma~\ref{lemma:tight}.
    Indeed, let $S \subseteq V(G)$ such that $|S| = \lceil n/2 \rceil$ and $\delta(G[S]) \geq \eta n p / 4$.
    Then Lemma~\ref{lem:pancyclicity} (applied with $\lambda = \eta/4$) implies that $G[S]$ is Hamiltonian, as required.
\end{proof}

\subsection{Vertex separation via sandwiching} \label{section:hamsandwich}

Here, we prove Proposition~\ref{prop:strengthen_regulardense}.
As mentioned before, its proof relies on a coupling argument which compares (relatively dense) random regular graphs with the corresponding binomial random graphs.

\begin{proof}[Proof of Proposition~\ref{prop:strengthen_regulardense}]
    Set $d_0 = d_0(n) = \lfloor (\log n)^8\rfloor$. If $d\in [(\log n)^2, d_0]$, 
    then a result of Kim and Vu~\cite[Theorem 2]{KV04} implies that $G(n,p)$ and $G(n,d)$ can be coupled for $p = d/2n$ so that a.a.s.\ $G(n,p)\subseteq G(n,d)$.
    Since $p = d/2n \gg \log n /n$, Proposition~\ref{prop:strengthen_thm1(i)} implies that a.a.s.\ $\sh(G(n,p)) = \lceil \log_2 n \rceil$.
    Since the parameter $\sh( \cdot)$ is monotone decreasing under addition of edges, we are done.
    
    If $d\ge d_0$, a recent result by Gao~\cite[Theorem~6]{Gao23} shows that the random graphs $G(n,d_0)$ and $G(n,d)$ can be coupled so that a.a.s.\ $G(n,d_0)\subseteq G(n,d)$, and we conclude from the argument given in the first paragraph.
\end{proof}

\subsection{\texorpdfstring{Hamiltonicity via expansion: proof of Proposition~\ref{prop:strengthen_thm2}}{}}

The proof of Proposition~\ref{prop:strengthen_thm2} will rely on a combination of Lemma~\ref{lemma:tight} and several preliminary results.
The first of them is a recent breakthrough of Dragani{\'c}, Montgomery, Munh\'a Correia, Pokrovskiy and Sudakov~\cite{DMCPS24} confirming that sufficiently good expanders are Hamiltonian.
A graph $H$ is said to be an \emph{$(\alpha,N)$-expander} if every set $S\subseteq V(H)$ of size $|S|\le N$ has at least $\alpha |S|$ neighbours in $V(H)\setminus S$.

\begin{theorem}[Theorem~1.3 in~\cite{DMCPS24}]\label{thm:DMCPS}
Fix any sufficiently large $\alpha > 0$, an integer $n\ge 3$ and an $(\alpha, n/2\alpha)$-expander $H$ on $n$ vertices. 
Suppose that, for every two disjoint vertex sets $X,Y\subseteq V(H)$ with sizes at least $n/2\alpha$, $H$ contains an edge with one endpoint in each of $X$ and $Y$. Then, $H$ is Hamiltonian.
\end{theorem}

We will also need a couple of auxiliary results for $d$-regular graphs.
For integers $n\ge 1$, $d\in [n-1]$ and a real number $\lambda > 0$, a graph is said to be an \emph{$(n,d,\lambda)$-graph} if it is $d$-regular, has $n$ vertices and the eigenvalues $\lambda_1 = d\ge \lambda_2\ge \ldots\ge \lambda_n$ satisfy $\max(|\lambda_2|, |\lambda_n|)\le \lambda$.
The next result is a consequence of a more general theorem due to Broder, Frieze, Suen and Upfal~\cite{BFSU98} who showed that the random regular graph $G(n,d)$ is typically an $(n,d,\lambda)$-graph for a suitably chosen $\lambda$, see also~\cite{BHKY20,CGJ18,Sar23} for more recent improvements.

\begin{theorem}[see Theorem~7 and Section 10.3 in~\cite{BFSU98}]\label{thm:G(n,d) is (n,d,l)}
There is a constant $C > 0$ such that, for every $d = d(n)\in [C, (\log n)^2]$, a.a.s.\ the random $d$-regular graph $G(n,d)$ is an $(n,d,\lambda)$-graph for $\lambda = 3\sqrt{d}$.
\end{theorem}

One of the main reasons to study $(n,d,\lambda)$-graphs are their good expansion properties. This fact is made precise by the next lemma, which is a combination of two results due to Pavez-Sign\'e~\cite{Pav23}.

\begin{lemma}[Corollary 3.7 and Lemma 3.8 in~\cite{Pav23}]\label{lem:Pav}
Fix integers $n,d,D\ge 1$ and a real number $\lambda > 0$ satisfying $2\lambda D \le d\le n-1$. Let $H$ be an $(n,d,\lambda)$-graph.
\begin{enumerate}
    \item[\emph{(a)}] For every pair of disjoint sets $X,Y$ of size $\lceil \lambda n/d\rceil$, $e(X,Y) \ge 1$.
    \item[\emph{(b)}] Suppose that a set $U\subseteq V(H)$ satisfies that $\delta(H[U])\ge 2\lambda D$. Then, every set $S\subseteq U$ of size at most $\lambda n/d$ satisfies that $|N(S)\cap U|\ge D |S|$.
\end{enumerate}
\end{lemma}
\noindent
Note that Lemma~3.8 in~\cite{Pav23} imposes the stricter assumption $2\lambda D < d$ but a straightforward analysis of the argument presented there confirms that the result holds as stated above.

The proof of Proposition~\ref{prop:strengthen_thm2} can now be deduced almost directly from the mentioned results and Lemma~\ref{lemma:tight}.

\begin{proof}[Proof of Proposition~\ref{prop:strengthen_thm2}]
Fix $\alpha$ as in Theorem~\ref{thm:DMCPS} and $C$ as in Theorem~\ref{thm:G(n,d) is (n,d,l)}, and define $D = \lceil \max(600\alpha^2,C) \rceil$.
Next, let $d = d(n)\in [D, (\log n)^2]$, and define $\lambda\in [3\sqrt{d},4\sqrt{d}]$ such that $d/2\lambda\in \mathbb N$. 
Then, by Theorem~\ref{thm:G(n,d) is (n,d,l)}, a.a.s.\ $G(n,d)$ is an $(n,d,\lambda)$-graph.
By Lemma~\ref{lem:Pav} (with $d/2\lambda$ playing the role of $D$), we get that $G(n,d)$ satisfies the assumptions of Theorem~\ref{thm:DMCPS} with $d/2\lambda > \alpha$ insteaf of $\alpha$, and this finishes the proof.
\end{proof}

\section{\texorpdfstring{The critical regime: proof of Theorem~\ref{thm:main_ER}(ii)}{}}\label{sec:4}

Our proof of the second part of Theorem~\ref{thm:main_ER} consists of two steps. 
Our first step (Proposition~\ref{prop:critical1}) exhibits the correct asymptotic expression of $\sp(G(n,p))$ when $np\ge (1-\eps)\log n$ except when $np = \log n + O(1)$, in which case the upper and the lower bounds are within a multiplicative factor of 2 from each other. 

\begin{proposition}\label{prop:critical1}
Fix $\eps > 0$ sufficiently small, $p = p(n)\in [0,1]$ satisfying $np\ge (1-\eps)\log n$ and $G\sim G(n,p)$. Then, a.a.s.\
\begin{equation}\label{eq:main}
(1-o(1))\max\left(\log_2 n, \frac{2n^2p\e^{-np}}{3}\right)\le \sp(G)\le (1+o(1))\left(\log_2 n + \frac{2n^2p\e^{-np}}{3}\right).
\end{equation}
\end{proposition}

\noindent
We note that, in the regime $n^2p\e^{-np} = O(\log\log n)$, our proof of Proposition~\ref{prop:critical1} guarantees the slightly stronger conclusion that a.a.s.\ $\sp(G) = \log_2 n + O(\log\log n)$, see \eqref{equation:whatweget}.

The next proposition closes the gap between the lower and the upper bound in~\eqref{eq:main} in the remaining case, i.e. when $np = \log n+O(1)$ (though the result holds for a slightly larger window around $\log n$).
	
\begin{proposition}\label{prop:maxcase}
Fix $p = p(n)$ satisfying $np = \log n \pm \tfrac{1}{2} \log \log n$ and $G\sim G(n,p)$. Then, a.a.s.
\begin{equation*}
\sp(G)\le (1+o(1))\max\left(\log_2 n, \frac{2n^2p\e^{-np}}{3}\right).
\end{equation*}
\end{proposition}

\subsection{Hamilton cycles in 2-cores}

For a graph $H$ and an integer $k\ge 2$, the $k$-core $\mathrm{Co}_k(H)$ is the unique maximal subgraph of $H$ with minimum degree $k$. 
The topic of cores of random graphs has been well-studied not without a reason: 
on the one hand, having a large well-connected subgraph with good expansion properties is practical in many situations 
and, on the other hand, the appearance of (non-empty) $k$-cores in random graphs exhibits a discontinous phase transition~\cite{PittelSpencerWormald1996}.

To prove Propositions~\ref{prop:critical1} and~\ref{prop:maxcase}, we need to understand how 2-cores of binomial random graphs behave with respect to Hamiltonicity.
One classical result guaranteeing the Hamiltonicity of the $k$-core of sufficiently dense binomial random graphs is due to {\L}uczak~\cite{Luc87}; in fact, he considered the existence of many edge-disjoint Hamilton cycles and his result is the best possible in this regard.
More precisely, denote by $\cM_{k}$ the set of all graphs containing $\lceil k/2\rceil$ edge-disjoint Hamilton cycles and, if $k$ is odd, an additional perfect matching that is also edge-disjoint from the latter cycles (and covers all vertices but one if $n$ is odd).
The following is a simplified version of \L uczak's result.

\begin{theorem}[see Theorem~3 in~\cite{Luc87}]\label{thm:Luc}
Fix $k\ge 2$ and let
\[p = \frac{\log n/(k+1) + k\log\log n + 2c_n}{n}.\] 
Let $G\sim G(n,p)$ be a binomial random graph.
Then,
\begin{equation*}
\lim_{n\to \infty} \mathbb P(\mathrm{Co}_k(G)\in \cM_{k}) = 
\begin{cases}
0, & \text{if } c_n\to -\infty,\\
1, & \text{if } c_n\to \infty.
\end{cases}
\end{equation*}
\end{theorem}
\noindent
In particular, for $k=2$, this result exhibits a (sharp) threshold for the property that the $2$-core of $G(n,p)$ is Hamiltonian.

\begin{remark}\label{rem:Luc}
For our desired application we will need to ensure that, with probability $1-o(1/(\log n)^2)$, there exists a Hamilton cycle in the 2-core of the random graph $G(n,p)$ when $np\ge (1/3+\eps)\log n$. While this quantitative bound on the probability of error does not follow from Theorem~\ref{thm:Luc}, it can be deduced by analysing {\L}uczak's original proof (which actually ensures a polynomial error bound). 
For simplicity and completeness, we provide an alternative argument using Lemma~\ref{lemma:hamcon2core} in Section~\ref{sec:5.3} going roughly as follows. 
First, we expose all edges of $G(n,p)$ except a single hidden one and show that the total variation distance between $G(n,p)$ and the exposed graph is $o(1/(\log n)^2)$. 
Then, we use that, by Lemma~\ref{lemma:hamcon2core}, with probability $1-o(1/(\log n)^2)$, all pairs of vertices but $o((n/\log n)^2)$ are endpoints of a Hamilton path in the 2-core of $G(n,p)$. Since the hidden edge is distributed uniformly at random among the missing ones, it completes a Hamilton cycle in the 2-core with the required probability.
\end{remark}

We will also need a variation of Lemma~\ref{thm:Luc} where, instead of a Hamilton cycle, we are interested in the existence of a Hamilton path which starts and ends in a prescribed pair of vertices.
Contrary to the quantification of the failure probability described in Remark~\ref{rem:Luc}, ensuring the additional flexibility for the pairs of endpoints of our Hamilton paths requires additional ideas.
The key statement we need is featured in the following lemma, whose proof we defer to Section~\ref{sec:5}. For a graph $G$, a vertex $v$ and an integer $D\ge 1$, we say that the vertex $v$ is \emph{$D$-far (in $G$)} if there is no vertex $w$ of degree at most $D$ and such that $\operatorname{dist}(v, w)\le 8$ (that is, the graph distance between $v$ and $w$ is at most 8).

\begin{lemma} \label{lemma:hamcon2core}
The following holds for all sufficiently small $\eps > 0$ and $\delta = \delta(\eps) > 0$.
Fix $G \sim G(n,p)$, $H = \Co(G)$ and suppose that $np\in [(1/3 + \eps) \log n, (1 - \eps) \log n]$.
Then, with probability $1-o(1/(\log n)^2)$, for every pair $x_1, x_2 \in V(H)$ such that
\begin{enumerate}
\item[\emph{(a)}] $\operatorname{dist}_H(x_1, x_2) \geq 8$, and
\item[\emph{(b)}] both $x_1$ and $x_2$ are $\delta np$-far,
\end{enumerate}
		$H$ contains a Hamilton path with endpoints $x_1$ and $x_2$.
\end{lemma}

\begin{remark}
Almost all pairs of vertices $x_1,x_2$ in $H$ satisfy the two assumptions required by Lemma~\ref{lemma:hamcon2core}.
However, the second assumption is introduced for the sake of a more compact argument only, and it is possible to remove it from the statement.
Indeed, for $np$ in the range indicated by Lemma~\ref{lemma:hamcon2core} and suitably small $\delta$, if the vertices $x_1,x_2$ are not $\delta np$-far, then one can follow short paths $P_1,P_2$ away from $x_1,x_2$ to $y_1,y_2$, respectively, with the property that $y_1,y_2$ satisfy both assumptions in the graph $H\setminus ((P_1\setminus \{y_1\})\cup (P_2\setminus \{y_2\})$.
Our proof (in particular, Definition~\ref{def:H*}) require minor modifications to derive the more general version, but we omit the details. 
\end{remark}

\subsection{\texorpdfstring{The regime $np\ge (1-\eps)\log n$: proof of Proposition~\ref{prop:critical1}}{}}\label{sec:crit_range}

For a real-valued finite-dimensional vector, we define its \emph{weight} as the sum of its coordinates.
For a constant $C \ge 1$, we set
\[\ell = \ell(n, C) = 
\left\lceil\frac{\log_2 n}{2} + C\log\log n\right\rceil\]
and denote by $\cV = \cV(\ell)$ the set of binary vectors of length $2\ell$ and weight $\ell$; in particular,
\[|\cV| = \binom{2\ell}{\ell} = \Theta(4^{\ell} /\ell^{1/2}) = \Theta(n (\log n)^{(\log 4)C-1/2})\gg n.\]

Fix $G\sim G(n,p)$. Our main task is to construct a family of paths separating the vertices of degree at least 2 in $G$.
To do this, we design a randomised procedure attributing vectors in $\cV$ to the vertices of $G$. 
First, we pick $N = n+\lceil n/(\log n)^3\rceil$ binary vectors in $\cV$ uniformly at random and independently.
From this collection, we delete the vectors that have another vector at Hamming distance at most $C_1$ from them, where $C_1\ge 2$ is an even positive integer to be chosen suitably large in the sequel.
In particular, turning less than $C_1/2$ ones to zeros in each of the surviving vectors still produces a family of distinct vectors with slightly smaller weights.

\begin{lemma}\label{lem:far}
Given $C \ge 3C_1\ge 12$, the number of deleted vectors is a.a.s.\ at most $N-n$.
\end{lemma}
\begin{proof}
For a fixed vector $x\in \{0,1\}^{2\ell}$ of weight $\ell$, the number of vectors of length $2\ell$, weight $\ell$ and at distance at most $C_1$ from $x$ is
\[\sum_{i=0}^{C_1/2} \binom{\ell}{i}^2\le \ell^{C_1}\le (\log n)^{C_1}.\]
Thus, Markov's inequality shows that the probability that some of these vectors is selected is bounded from above by 
\[ \frac{N(\log n)^{C_1}}{\tbinom{2\ell}{\ell}} = O \left( \frac{n (\log n)^{C_1}}{n (\log n)^{(\log 4)C - 1/2}} \right) = o(1/(\log n)^3).\] Then, since the expected number of deleted vectors is $o(n/(\log n)^3)$, Markov's inequality finishes the proof.
\end{proof}

At this point, if no more than $n$ of the selected vectors survive after the deletion, we give up on all remaining considerations; this only happens with probability $o(1)$ by Lemma~\ref{lem:far}.
Otherwise, out of the vectors that survive, we keep a uniformly random subset of $n$ vectors and distribute them again uniformly at random among the vertices of $G$.
For convenience, we identify the vertex set of $G$ with $[n]$ and denote by $x_i$ the binary vector associated to vertex $i$.
Then, for every $j\in [2\ell]$, we denote by $S_j$ the set of vertices $i\in [n]$ such that the $j$-th coordinate of $x_i$ is 1, which we also write as $x_{ij} = 1$.

\begin{lemma}\label{lem:couple}
Fix an integer $k\ge 1$. Then, with probability $1-o(1/\ell^{20})$,
\begin{itemize}
    \item for every distinct $j_1, j_2, \ldots, j_k\in [2\ell]$, $|V(G)\setminus (S_{j_1}\cup \dotsb \cup S_{j_k})|\le n/2^k + n/\log n$, and
    \item for every $j\in [2\ell]$, $||S_j| - n/2| \le n/\log n$.
\end{itemize}
\end{lemma}
\begin{proof}
Denote by $S_j'$ the set of vectors among the originally selected ones (that is, before the deletion) with 1 in coordinate $j$, and let $S' = S_1'\cup \dotsb \cup S_{2\ell}'$.
Then, given $k\ge 1$ and $j_1, j_2, \ldots, j_k\in [2\ell]$, 
$|S'\setminus (S_{j_1}'\cup \ldots\cup S_{j_k}')|$ is distributed as a binomial random variable with parameters $N$ and $1/2^k$. 
As a result, Chernoff's inequality implies that, with probability $1-\exp(-\Omega(N^{1/3})) = 1-o(1/\ell^{k+20})$, 
\begin{equation}\label{eq:complement}
|V(G)\setminus (S_{j_1}\cup \dotsb\cup S_{j_k})|\le |S'\setminus (S_{j_1}'\cup \dotsb\cup S_{j_k}')|\le N/2^k+N^{2/3}\le n/2^k+n/\log n.    
\end{equation}
A union bound over the complements of these events for all $k$-tuples of sets finishes the proof of the first statement.

For the second statement, on the one hand, we already have that a.a.s.\ $|S_j|\ge |V(G)|-(n/2 + n/\log n) = n/2-n/\log n$ for all $j\in [2\ell]$. On the other hand, for every $j\in [2\ell]$, Chernoff's inequality shows that, with probability $1-\exp(-\Omega(N^{1/3})) = 1-o(1/\ell^{21})$,
\[|S_{j}|\le |S_{j}'|\le N/2+N^{2/3}\le n/2+n/\log n.\]
Another union bound over the complements of the above $2\ell$ events finishes the proof of the second statement.
\end{proof}


\begin{corollary}\label{cor:couple}
Fix $\delta\in (0, 43/30)$ and suppose that $np\in [(2/3+\delta)\log n, 2.1\log n]$.
Then, a.a.s.\ the $2$-core of each of $G[S_j]$ for $j\in [2\ell]$ is Hamiltonian.
\end{corollary}
\begin{proof}
Condition on the second event in the statement of Lemma~\ref{lem:couple} and fix $j\in [2\ell]$.
Then, conditionally on $s = |S_j|$, $G[S_j]$ is a binomial random graph with distribution $G(s,p)$.
Since $|s - n/2| \le n/\log n$, we have that $sp\in [(1/3+\delta/3)\log s, 1.1\log s]$ and Remark~\ref{rem:Luc} implies that the 2-core of $G[S_j]$ is Hamiltonian with probability $1 - o(1/\ell)$.
Hence,
\begin{align*}
\mathbb P(\exists j\in [2\ell]: \Co(G[S_j])\text{ not Hamiltonian})
&\le 2\ell\, \mathbb P(\Co(G[S_1])\text{ not Hamiltonian}) = o(1),
\end{align*}
as desired.
\end{proof}

Given the Hamiltonicity of the 2-cores of $G[S_1],\ldots,G[S_{2\ell}]$, we would like to identify a simple local criterion allowing us to say if a vertex in $S_i$ belongs to the 2-core of $G[S_i]$.
Fortunately, in the regime $np\ge (1-\eps)\log n$ with $\eps > 0$ sufficiently small, such a criterion exists.
For a set $S\subseteq [n]$, denote by $V_S$ the subset of $S$ containing the vertices that either have at least $3$ neighbours in $S$ or have $2$ neighbours in $S$ of degree at least $2$ in $G[S]$.

\begin{lemma}\label{lem:deg 2 in core}
Fix a set $S\subseteq [n]$ of size $s\ge n/2-n/\log n$, a sufficiently small $\eps > 0$ and $np\ge (1-\eps)\log n$.
With probability $1-o(1/\ell)$, every vertex in $V_S$ is in the $2$-core of $G[S]$.
\end{lemma}

\begin{proof}
First of all, by combining Chernoff's inequality and a union bound, the probability that there is a set $U\subseteq S$ of size $|U|\ge s/3$ spanning less than $s$ edges in $G$ is bounded from above by
\[2^s \mathbb P(\mathrm{Bin}(\tbinom{|U|}{2},p)\le s) = 2^s \e^{-\Omega(s^2p)} = o(1/\ell).\]
In particular, with probability $1-o(1/\ell)$, the $2$-core of $G[S]$ contains at least $2s/3$ vertices.

To finish the proof we show that, with probability $1-o(1/\ell)$, there are no vertex sets $U\subseteq S$ of size $u = |U|\in [4,s/2]$ with the following properties:
\begin{itemize}
    \item $G[U]$ is a connected graph (so it contains at least $u-1\ge 3$ edges),
    \item $G[U]$ is either a connected component in $G$ or contains a vertex $v\in \Co(G)$ such that no edge connects a vertex in $U\setminus \{v\}$ to a vertex in $S\setminus U$.
\end{itemize}
The expected number of such sets is bounded from above by
\begin{align*}
\sum_{u=4}^{\lceil s/2\rceil} \binom{s}{u} u^{u-2} p^{u-1} (1-p)^{(u-1)(s-u)} \le\,
& 16n^4 p^3 (1-p)^{3(s-4)} 
 + 125 n^5 p^4 (1-p)^{4(s-5)}\\
&+ n \sum_{u=6}^{\lceil s/2\rceil} (\e sp)^{u-1} (1-p)^{(u-1)s/2}.    
\end{align*}
Again, since $\e sp (1-p)^{s/2}\ll 1$, the sum which is the last term of the right hand side is of order $O(n(sp)^5 \e^{-5sp/2}) = o(1/\ell)$ and, moreover, the two other terms are of order $o(1/\ell)$ as well.
Thus, with probability $1-o(1/\ell)$, by deleting the edges of the 2-core of $G[S]$, we remain with isolated vertices, edges and paths of length 2 intersecting $\Co(G[S])$ in at most one vertex, as desired.
\end{proof}

\begin{remark}\label{rem:isolated}
A similar (but a lot more immediate) first moment computation shows that:
\begin{itemize}
    \item if $np\ge c\log n$ with $c>1/2$, with probability at least $1-n^{-\eps}$ for some $\eps = \eps(c) > 0$, $G$ has a number of isolated vertices and a connected component containing all edges,
    \item if $np\ge c\log n$ with $c > 1/3$, with probability at least $1-n^{-\eps}$ for some $\eps = \eps(c) > 0$, $G$ has no connected component containing between 3 and $n/2$ vertices.
\end{itemize}
A more complete description can be found as Theorem~5.4 in~\cite{Bol01}.
\end{remark}

Now, to every vertex $i\in [n]$, we associate the vector $y_i$ of length $2\ell$ that contains 1 in coordinate $j\in [2\ell]$ if vertex $i$ belongs to the 2-core of $G[S_j]$, and 0 otherwise. 
Clearly $y_i$ is dominated by $x_i$ for every $i\in [n]$ in the sense that, if $y_{ij}=1$ for some $j\in [2\ell]$, then $x_{ij} = 1$ as well.
In fact, Lemma~\ref{lem:deg 2 in core} shows that a.a.s.\ the vectors $(y_i)_{i=1}^n$ are determined by the vectors $(x_i)_{i=1}^n$ and the graph $G$ in the following way:
\begin{enumerate}[\upshape{\textbf{A\arabic*}}]
    \item if $i$ has degree 0 or 1 in $G$, then $y_i$ is the all-zero vector of length $2\ell$,
    \item\label{item:CN2} if $i$ has degree at least 2 in $G$, then $y_{ij}=1$ if either there are neighbours $i_1, i_2, i_3$ of $i$ with $i, i_1, i_2, i_3\in S_j$, or there are neighbours $i_1, i_2$ of $i$ and neighbours $j_1\neq i$ of $i_1$ and $j_2\neq i$ of $j_2$ such that $i, i_1, i_2, j_1, j_2\in S_j$.
\end{enumerate}

Recall the constant $C_1$ in the description of the vector-attribution procedure.

\begin{lemma}\label{lem:high degree}
Fix a sufficiently small $\eps > 0$, $np\ge (1-\eps)\log n$ and $G\sim G(n,p)$. Suppose that $C_1 \ge 10^4$.
Then, a.a.s.\ for every vertex $i$ of degree at least $\log n/10^3$ in $G$, the vectors $x_i$ and $y_i$ are at Hamming distance at most $4000 < C_1/2$.
In particular, a.a.s.\ for every such vertex $i$ and every vertex $j\neq i$, $y_i\neq y_j$.
\end{lemma}
\begin{proof}
Fix a vertex $i\in [n]$ of degree at least $\log n/10^3$ and any $4000$ coordinates containing 1-bits in the vector $x_i$.
Suppose that, for each of these coordinates, there are at most 2 neighbours of $i$ whose corresponding vectors have 1-bits there. 
Then, there are at least $\log n/10^3 - 8000 > \log n/2000$ neighbours of $i$ that have 0-bits in each of these $4000$ coordinates.
Using Lemma~\ref{lem:couple}, we have that the probability of this event is at most
\[\left(\frac{n/2^{4000} + n/\log n}{n}\right)^{\log n/2000} = \left(\frac{1}{2^{4000}} + o(1)\right)^{\log n/2000} = o\left(\frac{1}{n^{1.1}}\right).\]
Taking a union bound over all at most $n$ vertices of degree at least $\log n/10^3$ and $O(\ell^{4000})$ choices of coordinates shows that a.a.s., for every vertex $i$ of degree at least $\log n/10^3$ in $G$, $x_i$ and $y_i$ differ in at most $4000$ coordinates. 
On this event, for every $j\neq i$, $x_j$ does not dominate $y_i$ since the distance between $x_i$ and $x_j$ is at least $C_1 > 2\cdot 4000$, which implies the second statement and finishes the proof.
\end{proof}

It remains to show that the family of Hamilton cycles in the 2-cores of $G[S_1], \ldots, G[S_{2\ell}]$ separates the vertices in $G$ of degree between 2 and $\log n/10^3$. To do this, we first assume to the end of this section that $np\le np_0 = \log n+\log\log n$: indeed, for $p > p_0$, we have that $n^2p\e^{-np} = o(\log n)$ and, since $\sp(\cdot)$ is a monotone parameter, showing that~\eqref{eq:main} holds with $p_0$ implies the statement for $p > p_0$.

We start by analysing the degree sequence of the graph $G$.
For all integers $k\in [0, \log n/10^3]$, let $X_k$ be the number of vertices of degree $k$ in $G$ and define 
\[n_k = n\cdot \binom{n-1}{k} p^k (1-p)^{n-k-1} = (1+o(1))n\frac{(np)^k}{k!} \e^{-np}.\]
In particular, $\mathbb E[X_k] = n_k$.
Let $K_0 = K_0(n,p)$ be any integer among $0,1,2$ with the property that $n_{K_0}\to \infty$ (note that $n_2\to \infty$ by our assumption that $p\le p_0$, so at least one choice for $K_0$ is possible). 
Also, define $K_1 = K_1(n) = \lfloor \log n/10^3\rfloor$ and the event
\begin{align*}
\cC = \{\forall k\in [K_0, K_1], |X_k - n_k| \le n_k/\log n_k\}.
\end{align*}

\begin{lemma}\label{lem:C}
$\cC$ is an a.a.s.\ event.
\end{lemma}
\begin{proof}
For every $k\in [K_0, K_1]$, we already saw that $\mathbb E[X_k] = n_k$. 
We now compute the second moment. 
We have that $\mathbb E[X_k^2]$ is given by
\[\mathbb E[X_k] + n(n-1)\left(p\binom{n-2}{k-1}^2 p^{2(k-1)}(1-p)^{2(n-k-1)} + (1-p)\binom{n-2}{k}^2 p^{2k} (1-p)^{2(n-k-2)}\right),\]
which rewrites as
\[\mathbb E[X_k] + O\left(\frac{k^2 \mathbb E[X_k]^2}{n^2p}\right) + \left(1 + O\left(pk+\frac{k^2}{n}\right)\right) \mathbb E[X_k]^2 = n_k + (1+O(pk))n_k^2.\]
As a result, using that $n_k + O(pk n_k^2) = n_k + o(n_k)$ for all $k\in [K_0, K_1]$, a second moment computation shows that
\[\mathbb P(|X_k-n_k|\ge n_k/\log n_k)\le \frac{n_k + o(n_k)}{(n_k/\log n_k)^2}\le \frac{(\log n_k)^3}{n_k}.\]
Moreover, for every $k\in [K_0, K_1-1]$, $n_{k+1}/n_k$ is bounded from below by a uniform constant larger than 1,
which implies that
\[\sum_{k=K_0}^{K_1} \frac{(\log n_k)^3}{n_k} = O\left(\frac{(\log n_{K_0})^3}{n_{K_0}}\right) = o(1),\]
as desired.
\end{proof}

The next lemma guarantees that, out of every pair of nearby vertices in $G(n,p)$ with $p$ slightly below $p_0$, the degree of at least one of them is not too small.

\begin{lemma}\label{lem:delete}
Fix $np\in [2\log n/3, np_0]$ and $G\sim G(n,p)$. 
Then, a.a.s.\ every pair of vertices $u,v$ at distance at most $10$ in $G$ satisfies $\deg(u)+\deg(v)\ge \log n/10$.
\end{lemma}
\begin{proof}
Fix a pair of vertices $u, v$ in $G$ and denote by $X_{u,v}$ the number of edges between $\{u,v\}$ and the rest of the graph. 
Also, denote by $\cE_{u,v}$ the event that there is a path of length at most 10 between $u$ and $v$ in $G$.
Setting $\zeta = \log n/10$, we have that
\begin{align}
\mathbb P(\cE_{u,v}\cap \{\deg(u)+\deg(v)\le \zeta\}) 
&\le\; \sum_{\ell=0}^{9} \binom{n}{\ell} p^{\ell+1} \sum_{i=0}^{\zeta-2} \binom{2(n-2)}{i} p^i (1-p)^{2(n-2)-i} \nonumber\\
&\le\; (2\log n)^9 p\sum_{i=0}^{\zeta-2} \binom{2(n-2)}{i} p^i (1-p)^{2(n-2)-i}\nonumber\\
&\le\; (2\log n)^9 p (\zeta-1) \binom{2(n-2)}{\zeta-2} p^{\zeta-2} (1-p)^{2(n-2)-(\zeta-2)}\nonumber\\
&\le\; \frac{1}{n} \left(\frac{2\e np}{\zeta-2}\right)^{\zeta-2}\e^{-2np + o(\log n)}.\label{eq:uv}
\end{align}
Setting $x = np/\log n$ and using that $(20\e x)^{1/10}\le \e^{x/3}$ for all $x\ge 2/3$ (or equivalently $\log(20\e x)/10\le x/3$)\footnote{See \url{https://www.desmos.com/calculator/woissjw5xv}.}, we have that \eqref{eq:uv} is bounded from above by $\e^{-5np/3}/n\le n^{-19/9}$. 
Thus, a union bound over all $O(n^2)$ vertex pairs proves the lemma.
\end{proof}

We are ready to show that the vertices whose degree in $G$ is sufficiently small are indeed separated by the family of Hamilton cycles we constructed in Corollary~\ref{cor:couple}.

\begin{lemma}\label{lem:small degree}
Fix a sufficiently small $\eps > 0$, $np\in [(1-\eps)\log n, np_0]$ and $G\sim G(n,p)$.
A.a.s.\ for all vertices $i,j$ with degree in the interval $[2,\log n/10^3]$, $y_i\neq y_j$.
\end{lemma}
\begin{proof}
Let us condition on the a.a.s.\ events from Lemmas~\ref{lem:C} and~\ref{lem:delete}.
Then, the number of pairs of vertices with degrees at most $K_1$ is 
\begin{align*}
O(n_{K_1}^2) = O\left(\frac{n^2 (np)^{2K_1}}{\e^{2np} (K_1!)^2}\right)
&= n^{2\eps+o(1)} \left(\frac{\e np}{K_1}\right)^{2K_1}\\
&= n^{2\eps+o(1)}\left(10^3\e\right)^{2\log n/10^3}.  
\end{align*}
Thus, since $\log(10^3\e) < 8$, for all sufficiently small $\eps$, the above expression is at most $n^{1/50}$.

Denote by $\cF$ the event that, for every pair of vertices $i,j$ with degree in the interval $[2,\log n/10^3]$, the vectors $x_i, x_j$ are at Hamming distance at least $\ell/2$. Our next step is to show that a.a.s.\ $\cF$ holds.
Using Stirling's formula, for any fixed binary vector $x\in \cV(\ell)$, the number of vectors of length $2\ell$, weight $\ell$ and at Hamming distance less than $\ell/2$ from $x$ is
\begin{align*}
\sum_{i=0}^{\lceil \ell/4\rceil} \binom{\ell}{i}^2 = O\left(\binom{\ell}{\lceil \ell/4\rceil}^2\right) 
&= O\left(\left(\frac{4^{1/4} (4/3)^{3/4}}{2}\right)^{2\ell}\cdot \binom{2\ell}{\ell}\right)\\
&= O\left((2/3^{3/4})^{2\ell}\cdot \binom{2\ell}{\ell}\right).
\end{align*}
Thus, by choosing uniformly at random $X_{K_1}\le 2n_{K_1}$ of the $n$ vectors remaining after the deletion, the probability to come across two vectors at Hamming distance at most $\ell/2$ from each other in the process is bounded from above by
\begin{align*}
O\left(\sum_{i=1}^{2 n_{K_1}} (i-1) (2/3^{3/4})^{2\ell} \binom{2\ell}{\ell}\cdot \frac{1}{n} \right)
&= n_{K_1}^2 (2/3^{3/4})^{\log_2 n} (\log n)^{O(1)}\\
&= n^{1/50} (2/3^{3/4})^{\log_2 n} (\log n)^{O(1)}.
\end{align*}
Using that $\log(3^{3/4}/2)/\log 2 > 1/10 > 1/50$, we deduce that $\mathbb P(\cF) = 1-o(1)$.

Finally, denote by $\cE$ the event from Lemma~\ref{lem:high degree}. 
Fix two vertices $i,j$ with degrees in the interval $[2, \log n/10^3]$. 
In order to have $y_i = y_j$, either the event $\cF$ fails or for each of given $2\lceil \ell/4\rceil$ coordinates $s$ where $x_i$ and $x_j$ differ, the vertex among $i,j$ in $S_s$ must satisfy condition~\ref{item:CN2}. 
To bound from above the probability of the latter event, we concentrate on two neighbours $i_1,i_2$ of $i$ and $j_1,j_2$ of $j$; note that each of them has at least $\log n/10 - \log n/10^3 - 1 \ge \log n/10^3$ neighbours.
Hence, unless $\cE$ fails, for each $a\in \{i_1,i_2,j_1,j_2\}$, the vectors $x_a$ and $y_a$ differ in at most $4000$ positions by Lemma~\ref{lem:high degree}.
As a result, unless $\cE\cap \cF$ fails, for~\ref{item:CN2} to hold, there must be at least $\zeta:= \lceil\ell/4\rceil-8000$ coordinates $l$ where $x_{il}=1=1-x_{jl}$ and $(x_{i_1 l}, x_{i_2 l})\neq (1,1)$, and at least $\zeta$ coordinates $l$ where $x_{jl}=1=1-x_{il}$ and $(x_{j_1 l}, x_{j_2 l})\neq (1,1)$.
However, the number of quadruplets of vectors in $\cV(\ell)$ satisfying this property for given $2\zeta$ coordinates is bounded from above by
\begin{align*}
\bigg(\sum_{a=0}^{\zeta} \binom{\zeta}{a} \binom{2\ell-\zeta}{\ell-a}\binom{2\ell-a}{\ell}\bigg)^2
&= \bigg(\binom{2\ell}{\ell} \sum_{a=0}^{\zeta} \binom{\zeta}{a} \binom{2\ell-\zeta}{\ell-a}\prod_{b=1}^{\ell}\frac{\ell+b-a}{\ell+b}\bigg)^2\\
&= O\bigg(\binom{2\ell}{\ell} \sum_{a=0}^{\zeta} \binom{\zeta}{a} \binom{2\ell-\zeta}{\ell-a}\exp\bigg(-a\bigg(\int_{\ell}^{2\ell} \frac{1}{x} \mathrm{d}x\bigg)\bigg)\bigg)^2\\
&= O\bigg(\binom{2\ell}{\ell} \sum_{a=0}^{\zeta} \binom{\zeta}{a} \binom{2\ell-\zeta}{\ell-a} 2^{-a}\bigg)^2. 
\end{align*}
Using the binomial formula, the latter expression can be rewritten as
\[O\bigg(\binom{2\ell}{\ell} \bigg(\frac{3}{2}\bigg)^{\zeta} 2^{2\ell-\zeta}\bigg)^2 = O\bigg(\ell \binom{2\ell}{\ell}^4 \bigg(\frac{3}{4}\bigg)^{2\zeta}\bigg).\]
Thus, using that $\ell = (1/2+o(1))\log_2 n$ and $\log(3/4)/(4\log 2) < -1/50$, we get that
\begin{align*}
\mathbb P(\exists i\neq j: &\deg(i),\deg(j)\in [2, \log n/10^3], y_i = y_j)\\
&\le \mathbb P(\overline{\cE}\cup \overline{\cF}) + O\bigg(n_{K_1}^2\cdot \binom{\lceil \ell/4\rceil}{\zeta}^2\cdot \ell \binom{2\ell}{\ell}^4 \bigg(\frac{3}{4}\bigg)^{2\zeta} \frac{1}{n^4}\bigg)\\
&= o(1) + n^{1/50 + \log(3/4)/(4\log 2) + o(1)} = o(1),
\end{align*}
as desired.
\end{proof}

At this point, the vertices of degree at least 2 in $G$ have been separated by $2\ell$ paths (or rather cycles). 
To finish the proof of Proposition~\ref{prop:critical1}, it remains to add a few extra paths to separate the vertices of degree $0$ and $1$, for which we do the following.
When $np\ge (1-\eps)\log n$ with $\eps\in (0,1/2)$, Remark~\ref{rem:isolated} implies that a.a.s.\ $G$ consists of isolated vertices and a large connected graph.
In this case, all leaves are contained in the giant component and can thus be separated with at most $\lceil 2X_1/3 \rceil$ paths (grouping the leaves in disjoint groups of three leaves and at most one group of size at most two, separating each group of three leaves $\{x, y, z\}$ with a $(x,y)$-path and $(y,z)$-path, and using at most two extra paths for the remaining group). 
Finally, we add one extra path for each isolated vertex.
Thus we have proven that, w.h.p.,
\begin{equation}
    \sp(G) \leq 2 \ell + X_0 + \lceil 2X_1/3 \rceil. \label{equation:whatweget}
\end{equation}

To obtain upper bound in~\eqref{eq:main} claimed in the statement, we consider two cases.
On the one hand, if $np-\log n\to \infty$, $n_0+n_1 = o(\log n)$ and Markov's inequality implies that a.a.s.\ there are $o(\log n)$ leaves and isolated vertices, i.e. $X_0 + X_1 = o(\log n)$.
Together with \eqref{equation:whatweget} we obtain the desired upper bound in~\eqref{eq:main}.
On the other hand, if $np-\log n-\log\log n\to -\infty$ and $np\to \infty$, $n_1\gg n_0$ and $n_1\to \infty$, so Lemma~\ref{lem:C} implies that a.a.s.\ the number of isolated vertices in $G$ is of smaller order than the number of leaves.
We also have that $X_1 = (1 + o(1)) n_1 = (1+o(1))n^2 p \e^{-np}$.
Again, together with \eqref{equation:whatweget} we the upper bound in~\eqref{eq:main},
and this completes the proof of Proposition~\ref{prop:critical1}.
	
\subsection{\texorpdfstring{Separating $G$ when $np = \log n + O(1)$: proof of Proposition~\ref{prop:maxcase}}{}}\label{sec:crit_point}

Our next goal is to close the gap between the lower and the upper bound in \eqref{eq:main} when $n^2p\e^{-np} = \Theta(\log n)$. More precisely, this section is dedicated to the proof of Proposition~\ref{prop:maxcase}.
	
Our construction will be similar to the one used in Section~\ref{sec:crit_range}.
Again, we use the random sets $S_j$ and the Hamilton paths $P_j$ in the $2$-cores of $G[S_j]$ to separate the vertices of degree at least $2$ in $G$.
However, instead of using extra $\left\lceil 2X_1/3 \right\rceil$ paths to separate the leaves, we will group these leaves into pairs $\{x_j, y_j\} \subseteq S_j$ and ensure that $P_j \subseteq G[S_j]$ can be extended to an $(x_j, y_j)$-path covering the $2$-core of $G[S_j]$ together with $x_j, y_j$.

\begin{proof}[Proof of Proposition~\ref{prop:maxcase}]
Set $t = t(n) = X_1$ to be the number of leaves in $G = G(n,p)$ and, since $n_1\to \infty$ in our regime of interest, recall that a.a.s.\
\begin{enumerate}[\textbf{B1}]
\item \label{assumption:leafs} $t = (1+o(1)) n^2p\e^{-np} \leq 2 (\log n)^2$.
\end{enumerate}
Moreover, by Lemma~\ref{lem:delete}, we have that a.a.s.
\begin{enumerate}[label=\textbf{B\arabic*}, resume]
\item \label{assumption:distancedegree} every pair of vertices $u,v$ at distance at most $10$ in $G$ satisfies $\deg(u)+\deg(v)\ge \log n/10$.
\end{enumerate}
Let $L = \{l_1, \dotsc, l_t\}$ be the set of leaves of $G$ and, for each $i\in [t]$, let $z_i$ be the unique neighbour of $l_i$. Consider the event
\begin{enumerate}[label=\textbf{B\arabic*}, resume]
\item \label{assumption:A3} for every pair of vertices $z',z''\in \{z_1, \dotsc, z_t\}$ and every $j\in [t]$ such that $z',z''\in S_j$, the assumptions of Lemma~\ref{lemma:hamcon2core} are satisfied for $z',z''$ and $G = G[S_j]$.
\end{enumerate}

\begin{claim}\label{cl:A3}
A.a.s.\ \emph{\ref{assumption:A3}} holds.
\end{claim}
\begin{proof}
First, by Lemma~\ref{lem:couple},  a.a.s.\ $|S_j| = n/2\pm n/\log n$ for all $j\in [2\ell]$. On this event, $|S_j|p \in [(1/3+\eps)\log |S_j|, (1-\eps)\log |S_j|]$ for sufficiently small $\eps > 0$, so the assumption on the edge density in Lemma~\ref{lemma:hamcon2core} is satisfied. Second, under the a.a.s.\ event \ref{assumption:distancedegree}, as parents of leaves, the vertices $z'$ and $z''$ are at graph distance at least 8 from each other in $G$ and the distance between them in $G[S_j]$ can only increase, which justifies assumption (a) in Lemma~\ref{lemma:hamcon2core}.
Finally, conditionally on the a.a.s.\ events \ref{assumption:distancedegree} and $\Delta(G)\le 20\log n$ (see Lemma~\ref{lem:max_deg}), for all $i\in [t]$, there are $O(t(\log n)^9) = O((\log n)^{10})$ vertices at distance at most 9 from the set $\{z_1,\ldots,z_t\}$ and each of them except $l_1,\ldots,l_t$ has degree at least $\log n/12$.
Denote $d_1 = \lceil \log n/12\rceil$.
Then, given $j\in [2\ell]$ and a vertex $v$ among the said $O((\log n)^{10})$ vertices, conditionally on the events $v\in S_j$, $|S_j| = n/2\pm n/\log n$ and $\deg_G(v)\ge d_1$, the probability that the degree of $v$ in $G[S_j]$ at most $d_1/3$ is bounded from above by
\begin{align*}
\sum_{i=1}^{\lfloor d_1/3\rfloor} \binom{|S_j|-1}{i} \binom{n-|S_j|}{d_1-i} \bigg{/} \binom{n-1}{d_1} = \frac{1}{(2+o(1))^{d_1}}\sum_{i=1}^{\lfloor d_1/3\rfloor} \binom{d_1}{i} = o((\log n)^{-11}),
\end{align*}
so a union bound over the $O((\log n)^{10})$ non-leaves at distance at most 9 from $\{z_1,\ldots,z_t\}$ and the $2\ell$ sets $S_1,\ldots, S_{2\ell}$ shows that a.a.s.\ assumption (b) of Lemma~\ref{lemma:hamcon2core} is satisfied for $\delta$ suitably small and finishes the proof.
\end{proof}

From now on, we condition on the a.a.s.\ events \ref{assumption:leafs} and \ref{assumption:distancedegree}. For each $i\in [t]$, let $a_i, b_i, c_i, d_i, e_i$ be five neighbours of $z_i$ distinct from $l_i$, and define the set 
\[T_i = \{l_i, z_i, a_i, b_i, c_i, d_i, e_i\}.\]
Note that $T_i \cap T_j = \emptyset$ whenever $i \neq j$, as otherwise the leaves $l_i$ and $l_j$ would be at graph distance at most $4$.

As before, we set $\ell = \ell(n,C)$ and randomly assign a binary vector $x_i$ of length $2 \ell$ and weight $\ell$ to each vertex $i \in [n]$. However, we use a slightly different attribution procedure.
\begin{enumerate}
    \item[(i)] The first step is the same: we sample a collection $\cV_N$ of $N$ vectors of length $2\ell$ and length $\ell$ independently and uniformly at random.
    \item[(ii)] Second, we choose a uniformly random subcollection $\cV_{7t}\subseteq \cV_N$ of size $|\cV_{7t}| = 7t$ and attribute the vectors in it uniformly at random to the vertices in $T := \bigcup_{i\in [t]} T_i$.
    \item[(iii)] Finally, from the remaining $N-7t$ vectors, we delete the ones at Hamming distance at most $C_1$ from some other vector in $\cV_{N}$. From the surviving vectors in $\cV_N\setminus \cV_{7t}$, if possible, we choose a uniformly random subset of size $n-7t$ and attribute them randomly to the vertices in $V(G)\setminus T$. If less than $n-7t$ vectors survive, the procedure is abandoned.
\end{enumerate}
In particular, since we delete a smaller number of vectors in this new procedure, Lemma~\ref{lem:far} implies that it a.a.s.\ outputs a valid result.

\begin{claim}\label{cl:attrib_procedures}
The original and the new vertex attribution procedures can be coupled in such a way that a.a.s.\ they both succeed and give the same output.
\end{claim}
\begin{proof}[Proof of Claim~\ref{cl:attrib_procedures}]
We reconstruct the original procedure using the new one as follows. Sample the sets $\cV_N$ and $\cV_{7t}$ as above. If all vectors in $\cV_{7t}$ are at Hamming distance more than $C_1$ from all other vectors in $\cV_N$, proceed as in part (iii) of the new procedure. If not, sample the $n$ vectors from $\cV_N$ independently of $\cV_{7t}$ as in the first procedure. 

Then, conditionally on deleting at most $N-n$ vectors in the first procedure (which happens a.a.s.\ by Lemma~\ref{lem:far}), by using Markov's inequality we ensure that, with probability at least $1-O(t(N-n)/N) = 1-o(1)$, no vector in $\cV_{7t}$ is at Hamming distance at least $C_1$ from all remaining ones in $\cV_N$, which provides the desired (a.a.s.\ successful) coupling.
\end{proof}

From this point in the proof, we work exclusively with the new vertex attribution procedure but stick, somewhat abusively, to the previously used notation.
Set $m = \min(\lfloor t/3\rfloor, \ell)$. 
Our goal is to allocate pairs of leaves to the sets $S_1,\ldots,S_{2m}$. More precisely, given integers $i,j \in [m]$, say that the pair $(i, j)$ is \emph{valid} if $T_{3i-2} \cup T_{3i-1} \subseteq S_{2j-1}$ and $T_{3i-1} \cup T_{3i} \subseteq S_{2j}$.
Note that, unless~\ref{assumption:A3} fails, for every valid pair $(i,j)$, we can find an $(l_{3i-2}, l_{3i-1})$-path containing all vertices in $\Co(G[S_{2j-1}])$ and an $(l_{3i-1}, l_{3i})$-path containing all vertices in $\Co(G[S_{2j}])$. 
Consider a random bipartite graph $\Lambda$ with parts $((T_{3i-2},T_{3i-1},T_{3i}))_{i\in [m]}$ and $(S_{2i-1},S_{2i})_{i\in [m]}$ where the edges correspond to the valid pairs.

\begin{claim}\label{cl:matching}
A.a.s.\ the graph $\Lambda$ has a perfect matching.
\end{claim}
\begin{proof}
On the one hand, the vectors attributed to the vertices in $T$ stochastically dominate a family of $7t$ independent binomial random vectors where every coordinate is equal to 1 with probability $1/3$: indeed, by Chernoff's inequality,
\[\mathbb P(\mathrm{Bin}(2\ell, 1/3)\ge \ell) = \exp(-\Omega(\ell)) = o(1/t).\]
However, in this binomial setting, the events $T_{3i-2} \cup T_{3i-1} \subseteq S_{2j-1}$ and $T_{3i-1} \cup T_{3i} \subseteq S_{2j}$ hold independently with probability $3^{-14}$.
Thus, the graph $\Lambda$ stochastically dominates a binomial random bipartite graph $G(m,m,3^{-28})$, which a.a.s.\ contains a perfect matching (see e.g.\ Section~7.3 in the book of Bollob\'as~\cite{Bol01} for this and stronger results).
\end{proof}

Finally, combining~\ref{assumption:leafs},~\ref{assumption:A3} and Claim~\ref{cl:matching} implies that we can spare $2m$ of the paths from the proof of Proposition~\ref{prop:critical1} (where we used distinct paths to separate the isolated vertices and the leaves from the remainder of the graph), thus improving the upper bound in~\eqref{equation:whatweget} by an additive factor of $2m = 2\min(\lfloor t/3\rfloor, \ell)$.
This implies that, w.h.p.,
\[ \sp(G) \leq 2 \ell + X_0 + \lceil 2 t / 3 \rceil - 2m = \max\{ 2\ell, \lceil 2 t / 3 \rceil \} + X_0 + O(1).\]
To get the bound claimed in Proposition~\ref{prop:maxcase}, we use that $2\ell = (1+o(1))\log_2 n$, and that w.h.p. in this probability range $X_0 = o(X_1)$ and $X_1 = (1+o(1))n^2 p \e^{np}$.
This finishes the proof.
\end{proof}

\section{\texorpdfstring{Hamilton-connectedness in $2$-cores}{}}\label{sec:5}
	
Given a graph $G$ and a positive real number $D$, we denote by $S_G(D)$ the set of vertices in $G$ of degree at most $D$ (sometimes abbreviated to $S(D)$, if the graph $G$ is clear from the context).
Recall that a vertex $v$ in a graph $G$ is $D$-far if there is no vertex $w$ in $S_G(D)\setminus \{v\}$ at distance at most $8$ from $v$.
	
To derive Lemma~\ref{lemma:hamcon2core}, we will prove that the existence of a $(x_1, x_2)$-Hamilton path in $\Co(G)$ is equivalent to the existence of a Hamilton cycle in a suitably defined auxiliary graph $H^\ast$.
Then, we apply the rotation-extension technique in $H^\ast$ in two steps.
First,~we will ensure that we can find a sparsified graph $F^\ast \subseteq H^\ast$ that is a $(2, n/7)$-expander.~Then, we will ensure that (with sufficiently high probability) the graph $H^\ast$ also contains a ``booster'' for any sparse expander $F \subseteq H^\ast$ (that is, an edge that extends the longest path or closes a Hamilton cycle when added to the graph) and conclude by iterative applications of this~fact.
For this part we use some insights of previous works on Hamiltonicity in random graphs, e.g.~\cite{Krivelevich2016}.
	
\subsection{\texorpdfstring{Constructing the auxiliary graph $H^\ast$}{}}

Fix an $n$-vertex graph $G$ and a positive real number $D$. We start by defining several properties that will be useful for our construction.
They will be a.a.s.\ satisfied for the random graph $G(n,\lambda/n)$ (for a valid choice of $\lambda$ and $D = \delta \lambda$ with $\delta > 0$ small).
\begin{enumerate}[\upshape{\textbf{C\arabic*}}]
\item \label{item:noclosesmallvertices} For every set $R \subseteq V(G)$ with $|R|\le 100$ and such that the graph $G[R]$ is connected, $|R \cap S(D)| \leq 2$,
\item \label{item:nocyclesmallvertices} for every set $R \subseteq V(G)$ with $|R|\le 100$ and such that $G[R]$ has at least as many edges as vertices, $R \cap S(D) = \emptyset$,
\item \label{item:nosmallcomponents} $G$ contains no connected components of size $k\in [3,n/2]$.
\end{enumerate}
	
The next lemma explains the effect of the above properties on the 2-core of $G$.
	
\begin{lemma} \label{lemma:degreeincore}
Suppose that $D \geq 3$ and $G$ satisfies \ref{item:noclosesmallvertices} and \ref{item:nosmallcomponents}. Let $H = \Co(G)$.
\begin{enumerate}[\upshape{(\roman*)}]
\item \label{item:degreeincore-lowbound} If $v \in V(H)$, then $d_H(v) \geq d_G(v) - 2$,
\item if $v \in V(H) \cap S_{G}(D)$, then $d_H(v) \geq d_G(v) - 1$, and
\item if $d_G(v) \geq 3$, then $v \in V(H)$.
\end{enumerate}
\end{lemma}
\begin{proof}
If $G$ contains no connected component of size more than $n/2$, then~\ref{item:nosmallcomponents} implies that all components contain one or two vertices and all points hold trivially.
Otherwise, $G$ consists of a single giant component with more than $n/2$ vertices, and tiny components with at most two vertices (which do not belong to the 2-core).
Hence, the largest component of $G$ consists of the $2$-core $H$ together with some trees intersecting $H$ at a single vertex (which we call the \emph{root} of the corresponding tree).
For every $v\in V(H)$, we denote by $T_v$ the tree containing $v$ in $G\setminus E(H)$ and observe that $d_H(v) = d_G(v) - d_{T_v}(v)$.
		
Fix any $r \in V(H)$. We claim that $|V(T_r)| \leq 3$.
To see this, take a deepest leaf $u$ in $T_r$ and let $v$ be its parent.
By~\ref{item:noclosesmallvertices}, $v$ is adjacent to at most two leaves.
If $v = r$, then this already shows that $|V(T_r)| \leq 3$, so assume that $v \neq r$.
Then, since $v\in S_G(3)\subseteq S_G(D)$,~\ref{item:noclosesmallvertices} implies that $v$ is adjacent to a single leaf and $\deg_G(v) = 2$.
Let $w$ be the parent of $v$.
If $w$ had another descendant apart from $v$ then, since $u$ is a deepest leaf, there would be another leaf at distance at most $4$ from $u$, thus contradicting \ref{item:noclosesmallvertices}.
If $w \neq r$, then $\deg_G(w) = \deg_{T_r}(w) = 2$ and $u, v, w\in S_G(D)$, a contradiction to~\ref{item:noclosesmallvertices}.
Hence, $w = r$ and $V(T_r) = \{u, v, w\}$, showing that $|V(T)| \leq 3$.
		
As a result, if $r$ is a vertex in $V(H)$, then $d_H(r) \geq d_G(r) - (|V(T_r)|-1)\ge d_G(r)-2$.
If moreover $r \in V(H) \cap S_G(D)$, then the above analysis implies that $|V(T_r)| \leq 2$, which improves the lower bound to $d_H(r)\ge d_G(r)-1$.
Finally, if $d_G(r) \geq 3$, then it must belong to the giant component of $G$ (since the remaining components contain at most two vertices).
Moreover, since the tree in $G\setminus E(H)$ containing $v$ has at most three vertices, $v$ also belongs to $H$, which finishes the proof.
\end{proof}
	
Fix $H = \Co(G)$ and two distinguished vertices $x_1, x_2 \in V(H)$. As explained before, to find a Hamilton $(x_1, x_2)$-path in $H$, we will look for a Hamilton cycle in an auxiliary graph $H^\ast$ which depends on $x_1, x_2$.
Before explaining how $H^\ast$ is constructed, we state a simple lemma that will ensure that the construction is well-defined.

\begin{lemma} \label{lemma:neighbouroutsidepath}
Set $S = S_G(D)$ and suppose that \ref{item:noclosesmallvertices}--\ref{item:nocyclesmallvertices} hold.
Suppose also that $P_1, P_2 \subseteq H$ are paths on at most $40$ vertices each such that $|V(P_1)\cap S|\ge 1$ and $|V(P_2)\cap S|\ge 2$.
Let $x_1, y_1$ and $x_2, y_2$ be the endpoints of $P_1, P_2$ respectively. Then, for both $i\in [2]$,
		\begin{enumerate}[\upshape{(\arabic*)}]
			\item \label{item:neighboursoutsidepath1} $N_H(x_1) \cup N_H(y_1)$ and $ N_H(x_2) \cup N_H(y_2)$ are disjoint,
            \item \label{item:neighboursoutsidepath1'} $N_H(x_i) \cap N_H(y_i)$ is empty unless $x_i, y_i$ have distance $2$ in $P_i$, in which case $N_H(x_i) \cap N_H(y_i)$ contains an unique vertex,
			\item \label{item:neighboursoutsidepath2} $N_H(x_i) \cup N_H(y_i)$ is disjoint from $V(P_{3-i})$,
			\item \label{item:neighboursoutsidepath3} $x_i, y_i$ have at least two neighbours in $H$ and exactly one neighbour in~$V(P_i)$.
\end{enumerate}
\end{lemma}
\begin{proof}
First, we derive~\ref{item:neighboursoutsidepath1}.
We show in fact that $N_H(x_1)$ is disjoint from $N_H(x_2)$; the other cases are checked in a similar way.
If $N_H(x_1) \cap N_H(x_2)$ contains a vertex $z$, then $R = P_1 \cup P_2\cup \{x_1z, x_2 z\}$ forms a connected graph on at most $81$ vertices which contains three vertices in $S$, thus contradicting \ref{item:noclosesmallvertices}.
Secondly, we check \ref{item:neighboursoutsidepath1'} for $i = 1$, the other case is analogous.
Suppose that $N_H(x_1) \cap N_H(y_1)$~contains a vertex $z$ such that $P_1 \neq x_1 z y_1$.
Then, $R = P_1\cup \{x_1 z, y_1 z\}$ forms a graph with at most $41$ vertices with $|V(R)|$ edges containing a vertex in $S$, which contradicts \ref{item:nocyclesmallvertices}.
The proof of~\ref{item:neighboursoutsidepath2} is similar to the proof of ~\ref{item:neighboursoutsidepath1}.
		
Finally, we show~\ref{item:neighboursoutsidepath3}.
We only check the statement for $x_1$ and $P_1$; the remaining assertions follow similarly.
On the one hand, $\deg_H(x_1) \geq 2$ because $x_1$ belongs to $H$, the $2$-core of $G$.
On the other hand, suppose that $x_1' \in V(P) \cap N_H(x_1)$ is not the neighbour of $x_1$ in $P_1$. Then, $R = P_1\cup \{x_1x'_1\}$ is a subgraph of $G$ on at most $40$ vertices, at least $|V(R)|$ edges and it contains a vertex in $S$, contradicting~\ref{item:nocyclesmallvertices}.
\end{proof}
	
We now construct the auxiliary graph $H^\ast$, which we call the \emph{reduced $(x_1, x_2)$-core}. 
We recall the distinguished vertices $x_1, x_2 \in V(H)$ and define the properties
\begin{enumerate}[\upshape{\textbf{C\arabic*}}]
\setcounter{enumi}{3}
\item \label{item:x1x2farapart} $\operatorname{dist}_H(x_1, x_2) \geq 8$,
\item \label{item:x1x2farfromsmall} $x_1, x_2$ are $D$-far in $H$.
\end{enumerate}
Roughly speaking, in our construction, iteratively and as long as possible, we replace short paths between low-degree vertices by single vertices of degree 2.
	
\begin{definition}[Reduced $(x_1, x_2)$-core]\label{def:H*}
Set $S = S_G(D)$ and suppose that \ref{item:noclosesmallvertices},\ref{item:nocyclesmallvertices}, \ref{item:x1x2farapart} and \ref{item:x1x2farfromsmall} are satisfied.		
Starting from $H$, we define the \emph{reduced $(x_1, x_2)$-core of $G$} as the graph $H^\ast = H^\ast(G, D, x_1, x_2)$ obtained by the following procedure.
Initially, we set $H_0 = H$ and update it iteratively as follows.
\begin{enumerate}
\item \label{item:reducedalg-step1} First, fix any neighbours $u_1,u_2$ of $x_1,x_2$, respectively. Remove $x_1,x_2$ from $H_0$ and add a new vertex $x_{12}$ whose unique neighbours are $u_1$ and $u_2$, thus forming $H_1$.
\item \label{item:reducedalg-step4} Suppose that $H_i$ is defined for some $i\ge 1$.
If there are no vertices $u, v\in S\cap V(H_i)$ joined by a $(u,v)$-path in $H$ of length at most $4$, set $H^{\ast} = H_i$. Otherwise, let $P_{uv}$ be such a path.
By Lemma~\ref{lemma:neighbouroutsidepath}\ref{item:neighboursoutsidepath3} and \ref{item:nocyclesmallvertices}, we can select distinct neighbours $u'$ of $u$ and $v'$ of $v$ in $H_i$.
Then, remove $V(P_{uv})$ from $H_i$, and replace it with a vertex $x_{uv}$ and edges $u'x_{uv}, v'x_{uv}$ to form $H_{i+1}$.
\end{enumerate}
Also, define $T^\ast$ as the set of all vertices added in the process, $S^\ast := S\cap V(H^\ast)$ and $U^\ast = V(H^\ast) \setminus (T^\ast \cup S^\ast)$. 
\end{definition}
	
Observe that the graphs induced by $S^\ast \cup U^\ast$ from $H$ and $H^*$ are identical. Also, notice that the construction removes $2$ vertices at the first step, and at most $5$ vertices at each iteration of the second step (of which there are at most $|S|/2$ many). In particular,
\begin{equation}\label{equation:estimatingnumberofverticesinauxgraph}
|S^\ast \cup U^\ast| \geq |V(H)| - 2 - 3 |S|.
\end{equation}
From this point on, we fix $S = S_G(D)$ and $H^* = H^*(G,D,x_1,x_2)$ with its corresponding sets $S^*$, $T^*$ and $U^*$.
The next lemma shows that, as claimed before, we can reduce the search of $(x_1, x_2)$-Hamilton paths in $H$ to the search of Hamilton cycles in $H^\ast$.
  
\begin{lemma}\label{lemma:hstarhamilton}
Suppose that $D \geq 3$. Then, if $H^\ast \neq \emptyset$ contains a Hamilton cycle, then $H$ contains a Hamilton $(x_1, x_2)$-path.
\end{lemma}
\begin{proof}
We work with the notation from Definition~\ref{def:H*}.
If $C_{i+1}$ is a Hamilton cycle of $H_{i+1}$ for some $i\ge 1$, then $(C_{i+1}\setminus x_{uv})\cup \{u'u, v'v\}\cup P_{uv}$ is a Hamilton cycle of $H_i$. Moreover, if $C_1$ is a Hamilton cycle of $H_1$, then $(C_1\setminus x_{12})\cup \{x_1u_1, x_2u_2\}$ is a $(x_1,x_2)$-Hamilton path of $H$. Since $H^*$ has a Hamilton cycle and $H^* = H_i$ for some $i\ge 1$, $H$ contains a Hamilton $(x_1, x_2)$-path, as desired.
\end{proof}
	
To argue that $H^\ast$ is Hamiltonian, we will show that there exists a sparse subgraph $F^\ast \subseteq H^\ast$ which inherits some expansion properties from $G$.
More precisely, for any vertex $v \in V(H^\ast)$, let $F(v) \subseteq E(H^\ast)$ be a subset of exactly $\min(d_{H^\ast}(v), D)$ many edges adjacent to $v$.
We define $F^\ast$ as the spanning subgraph of $H^\ast$ with edges $\bigcup_{v \in V(H^\ast)} F(v)$ and say that $F^\ast$ is a \emph{$D$-sparsification} of $H^\ast$.
Note that, by construction, $F^\ast$ has at most $Dn$ edges and $d_{F^\ast}(v) = d_{H^\ast}(v)$ if $d_{H^\ast}(v) \leq D$.
	
First, we study the expansion of subsets of $S^\ast \cup T^\ast \subseteq V(F^\ast)$ in an arbitrary $D$-sparsification.
For a vertex set $U$ in a graph, we call \emph{$U$-path} a path of length at most 4 between two vertices in $U$, and a \emph{$U$-cycle} a cycle of length at most 4 intersecting $U$.
	
\begin{lemma} \label{lemma:expansionHstar-small}
Suppose that $D \geq 5$ and that $x_1,x_2,G$ satisfy \ref{item:noclosesmallvertices}--\ref{item:nocyclesmallvertices} and \ref{item:x1x2farapart}--\ref{item:x1x2farfromsmall}. Fix a $D$-sparsification $F^\ast$ of $H^\ast$. Then, each of the following holds:
\begin{enumerate}[\upshape{(\roman*)}]
\item \label{item:expansionHstar-small-mindeg} $\delta(F^\ast) \geq 2$,
\item \label{item:expansionHstar-small-neighS} for each $X \subseteq S^\ast \cup T^\ast$, we have $|N_{F^\ast}(X)| \geq 2 |X|$, and
\item \label{item:expansionHstar-small-neighv} for each $x \in V(F^\ast)$, $|N_{H^\ast}(x) \cap N_{H^\ast}(S^\ast \cup T^\ast)| \leq 1$ and $|N_{F^\ast}(x) \cap N_{F^\ast}(S^\ast \cup T^\ast)| \leq 1$.
\end{enumerate}
\end{lemma}

\begin{proof}
We start by showing an auxiliary claim.
\begin{claim}\label{cl:H1}
$H^*$ contains no $(S^\ast \cup T^\ast)$-paths and $(S^\ast \cup T^\ast)$-cycles.
\end{claim}
\begin{proof}[Proof of Claim~\ref{cl:H1}]
Suppose first that there are $s_1, s_2 \in S^\ast \cup T^\ast$ at distance at most $4$. Consider a closest pair of such vertices (with respect to the graph distance in $H^*$) and let $Q$ be a shortest path between them.
As $H^*$ contains no paths of length at most 4 between vertices in $S^*\subseteq S$, at least one of $s_1, s_2$ does not belong to $S^\ast$.
Assume that $s_1 \in T^\ast$.
If $s_1, s_2 \neq x_{12}$, then $s_1$ (and possibly $s_2$) appeared after deletion of an $S$-path in $H$. 
This implies the existence of a triplet of vertices in $S$ contained in a connected subgraph of $H^*$ on less than 100 vertices, contradicting~\ref{item:noclosesmallvertices}.
Finally, if $s_1 = x_{12}$, this would mean that $x_1$ or $x_2$ has a vertex in $S$ at distance at most 8, contradicting~\ref{item:x1x2farfromsmall}.
The absence of $(S^\ast \cup T^\ast)$-cycles follows from~\ref{item:nocyclesmallvertices}.
\end{proof}
		
We come back to the proof of the lemma. First, we show \ref{item:expansionHstar-small-mindeg}. Fix a vertex $v$ in~$H^*$.
If $v \in T^\ast$, then $\deg_{F^\ast}(v) = \deg_{H^\ast}(v) = 2$ by construction.
If $v \in S^\ast$, then $N_H[v]$ and $N_{H^*}[v]$ must coincide as otherwise some neighbour of $v$ would be in an $S$-path $P \subseteq H\setminus \{v\}$, contradicting \ref{item:noclosesmallvertices}.
Since $H$ is the $2$-core of $G$, we conclude that $\deg_{F^*}(v) = \deg_{H^*}(v)\ge 2$ in this case.
Finally, if $v \in U^\ast$, by Lemma~\ref{lemma:degreeincore}\ref{item:degreeincore-lowbound} we have that $\deg_{H}(v) \geq D-2 \geq 3$.
We claim that $\deg_{H^\ast}(v) \geq \deg_H(v) - 1$.
Indeed, if $v$ is a neighbour of $x_1$ or $x_2$, this follows from~\ref{item:x1x2farfromsmall}.
If not, then~\ref{item:noclosesmallvertices} implies that~$v$ can have neighbours in at most one $S$-path $P$ in $H\setminus \{v\}$.
Moreover, if $P$ exists,~then $H[V(P)\cup \{v\}]$ contains no cycle by~\ref{item:nocyclesmallvertices}. 
Thus, $v$ has at most one neighbour in $P$ and $\deg_{H^\ast}(v) \geq \deg_H(v) - 1$.
In every case, we have that
\begin{equation}\label{eq:degF*}
\deg_{F^\ast}(v) = \min(\deg_{H^\ast}(v), D) \geq D - 3 \geq 2.
\end{equation}
		
To prove \ref{item:expansionHstar-small-neighS}, it is sufficient to combine \ref{item:expansionHstar-small-mindeg} and the fact that, by Claim~\ref{cl:H1}, the vertices in $S^*\cup T^*$ are pairwise at distance at least three in $H^*$.
Finally,~\ref{item:expansionHstar-small-neighv} also follows from Claim~\ref{cl:H1}: indeed, if a vertex $v$ had two neighbours in $N_{H^*}[S^*\cup T^*]\supseteq N_{F^*}[S^*\cup T^*]$, then this would either produce an $(S^*\cup T^*)$-path or an $(S^*\cup T^*)$-cycle in $H^*$, both excluded~by~Claim~\ref{cl:H1}.
\end{proof}
	
The previous lemma shows that subsets of $S^\ast \cup T^\ast$ expand in any sparsification $F^\ast$ of $H^*$. Now, we need to understand the subsets of $U^\ast$. We consider the following properties:
\begin{enumerate}[\upshape{\textbf{C\arabic*}}]
\setcounter{enumi}{5}
\item \label{item:sparsification-ii} every set $X \subseteq V(G)$ of size $|X| \leq n/(\log n)^{1/2}$ spans at most $(\log n)^{3/4}|X|$ edges in $G$,
\item \label{item:sparsification-iii} for every pair of disjoint sets $X, Y\subseteq V(G)$ of sizes $|X| \leq n/(\log n)^{1/2}$ and $|Y| \leq (\log n)^{1/4} |X|$, the number of edges in $G$ between $X$ and $Y$ is at most $D|X|/2$, and
\item \label{item:suitable-7} for every pair of disjoint sets $X, Y\subseteq V(G)$ of size $\lceil n/(\log n)^{1/2} \rceil$, the number of edges in $G$ between $X$ and $Y$ is at least $n/6$.
\end{enumerate}
Our next objective is to prove that if the original graph $G$ satisfies \ref{item:sparsification-ii}--\ref{item:suitable-7}, then there is \emph{some} $D$-sparsification $F^\ast \subseteq H^\ast$ where sets of large-degree vertices in $F^\ast$ expand.
Recall that the sets $S^*$, $T^*$ and $U^*$ partition the vertices of $H^*$.
	
\begin{lemma} \label{lemma:Dsparsification-exists}
Fix $\delta > 0$ suitably small.
Suppose that $D \geq \delta \log n$ and $G$ satisfies \ref{item:noclosesmallvertices}--\ref{item:suitable-7}, $\Delta(G) \leq 12 \log n$ and $|U^\ast| \geq 5n/6$.
Then, there exists a $D$-sparsification $F^\ast \subseteq H^\ast$~such~that:
\begin{enumerate}[\upshape{\textbf{D\arabic*}}]
\item\label{item:sparsyjoined} for every pair of disjoint sets $X, Y \subseteq U^\ast$ of size $\lceil n/(\log n)^{1/2} \rceil$, $e_{F^\ast}(X,Y)\ge 1$, and
\item\label{item:H2} for every set $X \subseteq U^\ast$ of size at most $n/7$, $|N_{F^\ast}(X) \cap U^\ast| \geq 4 |X|$.
\end{enumerate}
\end{lemma}
\begin{proof}
We define $F^*$ as a random $D$-sparsification of $H^\ast$. More precisely, for each vertex $v \in F^\ast$, select $F(v)$ as a set of $\min(D, \deg_{H^\ast}(v))$ edges adjacent to $v$ chosen uniformly at random, and let $F^\ast$ be the graph formed by taking the union of $F(v)$ for all $v \in V(H^\ast)$.

First, we show that~\ref{item:sparsyjoined} holds a.a.s.\ (over the random choice of $F^\ast$).
Fix two disjoint sets $X,Y\subseteq U^*$ of size $k := \lceil n / (\log n)^{1/2} \rceil$. 
By definition of $U^*$, we have that $H^\ast[X,Y] \cong G[X,Y]$ and, by \ref{item:suitable-7}, there are at least $n/6$ edges in $H^\ast[X,Y]$. 
Given $x \in X$, let $d_x := e_{H^\ast}(x, Y)$ be the number of neighbours of $x$ in $Y$.
Then, $x$ has no neighbour in $F(x)\cap Y$ with probability at most
\[\binom{\Delta(G) - d_x}{D}\bigg{/}\binom{\Delta(G)}{D} \leq \left( \frac{\Delta(G) - d_x}{\Delta(G)} \right)^D \leq \e^{- d_x D / \Delta(G)} \leq \e^{ - d_x \delta / 12}.\]
As a result, the probability that $F^\ast[X,Y]$ has no edges is at most
\[\prod_{x \in X} \e^{- d_x \delta / 12} = \e^{- e(H^\ast[X,Y]) \delta / 12} \leq \e^{ - \delta n/72}.\]
Finally, a union bound over the $\binom{n}{k}^2 \leq (n\e/k)^{2k} = \e^{o(n)}$ choices of sets $X,Y$ shows that~\ref{item:sparsyjoined} holds a.a.s.\ over the random choice of $F^*$.

We turn to~\ref{item:H2}. Fix $F^*$ satisfying~\ref{item:sparsyjoined} and set $F' = F^\ast[U^\ast]$. As $F'\subseteq H^\ast[U^\ast] \cong G[U^\ast]$, \ref{item:sparsification-ii}--\ref{item:sparsification-iii} imply the following properties:
\begin{enumerate}[\upshape{\textbf{D\arabic*}}]
\setcounter{enumi}{2}
\item \label{item:hsparse} every set $X \subseteq V(G)$ of size $|X| \leq n/(\log n)^{1/2}$ spans at most $(\log n)^{3/4}|X|$ edges~in~$F'$,
\item \label{item:hbipsparse} for every pair of disjoint sets $X, Y\subseteq V(G)$ of sizes $|X| \leq n/(\log n)^{1/2}$ and $|Y| \leq (\log n)^{1/4} |X|$, the number of edges in $F'$ between $X$ and $Y$ is at most $D|X|/2$.
\end{enumerate}

Fix a set $X\subseteq U^*$ of size $|X|\le n/7$.
We will show that $|N_{F'}(X)| = |N_{F^\ast}(X) \cap U| \geq 4 |X|$ by considering two cases.
Suppose first that $|X| \geq \lceil n/(\log n)^{1/2} \rceil$ and let $Y = U^\ast \setminus N_{F'}[X]$. 
Since $e_{F'}(X,Y) = 0$, by \ref{item:sparsyjoined}, we must have that $|Y| < \lceil n/(\log n)^{1/2} \rceil$.
Hence,
\[ |N_{F'}(X)| > |U^\ast| - |X| - \lceil n/(\log n)^{1/2} \rceil \geq |U^\ast| - \frac{n}{6} \geq \frac{4n}{6} \geq 4 |X|,\]
which shows our claim when $X$ is large.

If $|X| < \lceil n/(\log n)^{1/2} \rceil$, set $Z = N_{F'}(X)$ and suppose for contradiction that $|Z| < 4|X|$.
For every $v\in U^*$, by Lemma~\ref{lemma:expansionHstar-small}\ref{item:expansionHstar-small-neighv}, $|N_{H^\ast}(v) \cap (S^\ast \cup T^\ast)| \leq 1$, which together with~\eqref{eq:degF*} gives that $\deg_{F'}(v) \geq \deg_{F^\ast}(v) - 1 \geq D-4$. Using the latter observation and~\ref{item:hsparse}, we obtain that
\[e_{F'}(X,Z) = \bigg(\sum_{v\in X} \deg_{F'}(v)\bigg) - 2e(F'[X])\ge (D-4)|X| - 2(\log n)^{3/4}|X| > D|X|/2,\]
which contradicts~\ref{item:sparsification-iii}. This ends the proof of~\ref{item:H2}.
\end{proof}
	
Using the tools we developed so far, our next lemma shows that $H^\ast$ contains a connected and sparse expander.
	
\begin{lemma} \label{lemma:existenceexpander}
Fix $\delta > 0$ suitably small.
Suppose that $D \geq \delta \log n$ and $G$ satisfies \ref{item:noclosesmallvertices}--\ref{item:suitable-7}, $\Delta(G) \leq 12 \log n$ and $|U^\ast| \geq 11n/12$.
Then, there exists some $D$-sparsification $F^\ast \subseteq H^\ast$ which is a spanning connected $(2, n/7)$-expander.
\end{lemma}
\begin{proof}
Let $F^*$ be the $D$-sparsification of $H^\ast$ given by Lemma~\ref{lemma:Dsparsification-exists} and satisfying~\ref{item:sparsyjoined} and~\ref{item:H2}. 
Fix any non-empty set $X\subseteq V(F^*)$ of size at most $n/7$ and define
$X_1 = X \cap (S^\ast \cup T^\ast)$ and $X_2 = X \cap U^\ast$.
Then, Lemma~\ref{lemma:expansionHstar-small}\ref{item:expansionHstar-small-neighv} implies that $|N_{F^\ast}(X_2) \cap N_{F^\ast}[X_1]| \leq |X_2|$, so we have
\begin{align*}
|N_{F^\ast}(X)|
& = |N_{F^\ast}(X_1) \setminus X_2 | + |N_{F^\ast}(X_2) \setminus N_{F^\ast}[X_1]| \\
& \geq 2 |X_1| - |X_2| + |N_{F^\ast}(X_2)| - |X_2| \geq 2 |X_1| + 2|X_2| = 2|X|,
\end{align*}
showing that $F^*$ is indeed an $(2,n/7)$-expander.
		
It remains to check that $F^\ast$ is connected.
Suppose otherwise and let $Y$ be the vertex set of a smallest component in $F^\ast$; in particular, $|Y| \leq n/2$.
Since $F^*$ is a $(2,n/7)$-expander, we must have that $|Y|\ge n/7$ and, since $U^\ast \geq 11n/12$, both $Y \cap U^\ast$ and $U^\ast \setminus Y$ contain at least $n/12$ vertices.
However, \ref{item:sparsyjoined} implies that there must exist an edge between $Y \cap U^\ast$ and $U^\ast \setminus Y$, contradicting the fact that $Y$ spans a connected component of $F^*$, as desired.
\end{proof}

Given a graph $G$ and a vertex pair $e = \{u,v\} \subseteq V(G)$, we say that $e$ is a \emph{booster for $G$} if either $G\cup e$ is Hamiltonian, or the length of its longest path is larger than that of $G$.
Our last goal in this section is to show that every expander in $H^\ast$ must contain many boosters inside the set $U^\ast$.
As a preliminary step, we use the following result appearing as Lemma~16 in~\cite{Mon18} in the particular case $c = 1/5$; the proof of the more general case remains unchanged and is therefore omitted.

\begin{lemma}[see Lemma~16 in~\cite{Mon18}]\label{lem:boosters}
Fix $c\in (0,1/5]$.
Then, every $(2, cn)$-expander $H$ contains at least $c^2 n^2/2$ boosters.
\end{lemma}

\begin{lemma} \label{lemma:therearemanyboosters}
Suppose that $|U^*| \geq 499n/500$.
Then, for every spanning connected $(2, n/7)$-expander $F^\ast \subseteq H^\ast$,
there are at least $n^2/200$ boosters for $F^\ast$ inside $U^*$.
\end{lemma}
\begin{proof}
By Lemma~\ref{lem:boosters}, $F^*$ must contain at least $n^2/100$ boosters. Since the number of edges incident to $S^*\cup T^*$ is at most $\tbinom{n}{2} - \tbinom{|U^*|}{2}\le n^2/200$, at least $n^2/200$ of the boosters are contained in $U^*$, as desired.
\end{proof}
	
By this last lemma, to prove that $H^\ast$ is Hamiltonian, it is enough to show that $G$ satisfies \ref{item:noclosesmallvertices}--\ref{item:nosmallcomponents}, \ref{item:sparsification-ii}--\ref{item:suitable-7} and, in addition, for every sufficiently sparse spanning connected expander $F \subseteq H$, one of the many boosters for $F$ is already contained in $G$.
It turns out that all of those properties are satisfied with high probability in $G(n,p)$, as we will now check.
	
\subsection{\texorpdfstring{Verification of the expansion properties of the $2$-core}{}}
In the previous subsection, we gathered several properties that guarantee good expansion of the $2$-core and of the reduced $2$-core.
We now show that they typically hold in $G(n,p)$ for a suitable choice of $p$.

\subsubsection{\texorpdfstring{Checking \ref{item:noclosesmallvertices}--\ref{item:nosmallcomponents}, \ref{item:sparsification-ii}--\ref{item:suitable-7} and the assumptions of Lemmas~\ref{lemma:Dsparsification-exists},~\ref{lemma:existenceexpander} and~\ref{lemma:therearemanyboosters}}{}}
	
\begin{lemma} \label{lemma:components-small}
Fix sufficiently small $\eps > 0$, $\delta = \delta(\eps) > 0$, $p = p(n)$ with $np \geq (1/3 + \eps) \log n$, $D = \delta np$ and $G \sim G(n,p)$.
Then, $G$ satisfies \ref{item:noclosesmallvertices}--\ref{item:nosmallcomponents} with probability at least $1-n^{-\varepsilon}$.
\end{lemma}
\begin{proof}
First, we check~\ref{item:noclosesmallvertices}.
For $k\in [3,100]$, let $X_k$ be the number of connected $k$-vertex subgraphs of $G$ with at least three vertices of degree at most $D$. Using that there are $k^{k-2}$ trees on $k$ vertices, we have
\begin{align*}
\mathbb{E}[X_k]\leq \binom{n}{k} k^{k-2} p^{k-1} k^3 \bigg(\sum_{j=0}^{D} \binom{n}{j} p^j (1 - p)^{n - k - j}\bigg)^3.
\end{align*}
Moreover, using the standard inequality $\tbinom{n}{a}\le (\e n/a)^a$ for any $a\in [n]$ and the fact that the general term in the above sum is maximised for $j=D$, we obtain that
\begin{align*}
\mathbb{E}[X_k]\leq \bigg(\frac{\e n}{k}\bigg)^k (pk)^k \frac{k}{p}\cdot (D+1)^3 \bigg(\frac{\e n}{D}\bigg)^{3D} p^{3D} \e^{-3np}\le n (\log n)^{O(1)} \bigg(\frac{\e}{\delta}\bigg)^{3\delta np} \e^{-3np}.
\end{align*}
Since $(\e/\delta)^{\delta}$ can be arbitrarily close to 1 by choosing $\delta$ small and $3np = (1+3\eps)\log n$, for suitably chosen $\delta$, we get $\mathbb{E}[X_k]\leq n^{-2\eps}/100$.
Using Markov's inequality, we obtain that
\[ \mathbb{P}(X_3+\dotsb+X_{100} \ge 1) \leq \sum_{k=3}^{100} \mathbb{E}[X_k] \leq n^{-2 \eps}.\]
		
We turn to~\ref{item:nocyclesmallvertices}.
Similarly, for $k\in [3, 100]$, let $Y_k$ be the number of $k$-vertex subgraphs of $G$ with at least as many edges as vertices and containing a vertex of degree at most $D$. Then, similar reasoning as before shows that
\begin{align*}
\mathbb{E}[Y_k]
& \leq \binom{n}{k} \binom{k}{2}^k p^k\cdot k \bigg( \sum_{j=0}^{\delta np} \binom{n}{j} p^j (1 - p)^{n - k - j} \bigg) \le (\log n)^{O(1)} \bigg(\frac{\e}{\delta}\bigg)^{\delta np} \e^{-np}\le \frac{n^{-2\eps}}{100},
\end{align*}
where the last inequality holds for sufficiently small $\delta$.
Once again, Markov's inequality implies that $\mathbb{P}(Y_3 + \dotsb + Y_{100}\ge 1)\leq n^{-2\eps}$. Finally, \ref{item:nosmallcomponents} holds with polynomially small error probability by Remark~\ref{rem:isolated} and, up to adjusting the value of $\eps$, a union bound finishes the proof of the lemma.
\end{proof}
	
\begin{lemma}\label{lemma:sparseandjoined}
Fix sufficiently small $\eps, \delta > 0$, $p = p(n)$ with $np\in [(1/3 + \eps) \log n, 2\log n]$, $D = \delta np$ and $G \sim G(n,p)$.
Then, $G$ satisfies \ref{item:sparsification-ii}--\ref{item:suitable-7} with probability $1-o(1/(\log n)^2)$.
\end{lemma}
\begin{proof}
First, we check \ref{item:sparsification-ii}.
Fix $k\leq k_{\max} := \lfloor n/(\log n)^{1/2}\rfloor$ and $\beta = (\log n)^{3/4}$.
If $k\le \beta$, then no set on $k$ vertices can span $\beta k$ or more edges. 
Suppose that $k\in [\beta,k_{\max}]$ and let $X$ be a set of $k$ vertices.
The probability that $e(G[X]) \geq \beta k$ is at most
\[\binom{k(k-1)/2}{\beta k} p^{\beta k} \leq \left( \frac{\e k(k-1)p}{2\beta k} \right)^{\beta k} \leq \left( \frac{\e kp}{2\beta} \right)^{\beta k}.\]
Then, a union bound over the $\binom{n}{k}\le (\e n/k)^k$ sets of size $k$ and the (less than) $k_{\max}$ possible values for $k$ shows that~\ref{item:sparsification-ii} fails with probability at most
\[\sum_{k=\lceil\beta\rceil}^{k_{\max}} \binom{n}{k} \bigg(\frac{\e kp}{2\beta} \bigg)^{\beta k}\le \sum_{k=\lceil\beta\rceil}^{k_{\max}} \bigg(\frac{\e n}{k}\bigg(\frac{\e kp}{2\beta}\bigg)^3 \bigg)^{k}\le \sum_{k=\lceil\beta\rceil}^{k_{\max}} \bigg(\frac{1}{\log n}\bigg)^{k/5} = o\bigg(\frac{1}{(\log n)^2}\bigg).
\]

We turn to~\ref{item:sparsification-iii}. Fix $k\in [k_{\max}]$ and disjoint sets $X,Y$ of sizes $k$ and $\lfloor k (\log n)^{1/4}\rfloor$, respectively.
If $|Y|\le D/2$, then the number of edges between $X$ and $Y$ is clearly at most $D|X|/2$.
In general, $G[X,Y]$ has more than $D|X|/2 = \delta npk/2$ edges with probability~at~most
\[ \binom{|X||Y|}{\delta npk / 2} p^{\delta npk /2} \leq \left( \frac{2\e p|X||Y|}{\delta npk} \right)^{\delta npk/2}\le \left( \frac{2\e k(\log n)^{1/4}}{\delta n} \right)^{\delta npk/2}.\]
Set $k_y := \lfloor k(\log n)^{1/4}\rfloor$.
Then, for a fixed $k\in [k_{\max}]$, the probability that \ref{item:sparsification-iii} fails for some choice of sets $X,Y$ with prescribed sizes is dominated by
\begin{align*}
\binom{n}{k} \binom{n}{k_y} \bigg( \frac{2\e k(\log n)^{1/4}}{\delta n} \bigg)^{\delta npk/2}
& \leq \bigg( \frac{\e n}{k} \bigg)^k \bigg( \frac{\e n}{k_y} \bigg)^{k_y} \bigg( \frac{2\e k(\log n)^{1/4}}{\delta n} \bigg)^{\delta npk/2}\\
& \leq \bigg( \frac{\e n (\e n)^{(\log n)^{1/4}} (2 \e k (\log n)^{1/4})^{\delta np/2}}{k (k (\log n)^{1/4})^{(\log n)^{1/4}} (\delta n)^{\delta np/2}} \bigg)^k \\
& \leq \bigg(\e^{o(\log n)} \bigg(\frac{2\e (\log n)^{1/4}}{\delta}\bigg)^{\delta np/2} \bigg(\frac{k}{n}\bigg)^{\delta np/3}\bigg)^k \le \frac{1}{2^{\delta npk}}.
\end{align*}
A union bound over all values of $k\in [k_{\max}]$ shows that \ref{item:sparsification-iii} fails with probability~at~most
\[\sum_{k \geq 1} 2^{-\delta npk} = o(1/(\log n)^3).\]
	
To confirm~\ref{item:suitable-7}, fix $k = \lceil n/(\log n)^{1/2} \rceil$ and two disjoint vertex sets $X, Y$ of size $k$. Then, the expected number of edges in $G[X,Y]$ is
$\mu := k^2 p\geq n/3$. Moreover, by Chernoff's bound, the probability that $e(G[X,Y]) \leq n/6$ is at most $\exp(-\mu/8)\leq \exp(-n/24)$.
A union bound over the $\binom{n}{k}^2 \leq (\e n/k)^{2k} = \e^{o(n)}$ choices for $X,Y$ shows that~\ref{item:suitable-7} fails with probability $o(1/(\log n)^2)$ and finishes the proof.
\end{proof}

\begin{lemma}\label{lemma:n^11/12}
Fix sufficiently small $\eps, \delta > 0$, $p = p(n)$ with $np\in [(1/3 + \eps) \log n, 2\log n]$, $D = \delta np$ and $G \sim G(n,p)$.
Then, $\Delta(G) \leq 12 \log n$ and $|S_G(D)| \leq n^{11/12}$ with probability $1-o(1/(\log n)^2)$.
\end{lemma}
\begin{proof}
The first assertion follows from Lemma~\ref{lem:max_deg}.
For the second assertion, note that the expected degree of each vertex is $(n-1)p\approx np$ and, by Chernoff's bound, for all suitably small $\delta$, the probability that some vertex has degree at most $\delta np$ is dominated by $\exp(-np/3) \leq n^{-1/9}$. 
Hence, the expected size of $S_G(D)$ is at most $n^{8/9}$, and Markov's inequality finishes the proof of the second assertion.
\end{proof}

\subsubsection{Boosters of sparse spanning expanders}

We turn to showing that typically, for every sparse spanning connected expander $F \subseteq H$, one of the many boosters for $F$ is already contained in $G$.

\begin{definition}[Boosterable families]\label{def:bf}
Fix a family $\cF$ of pairs of graphs $(F, B)$ on a fixed set of $n$ vertices.
We say that $\mathcal{F}$ is a \emph{$(\delta, \alpha, M)$-boosterable family} if each of the following properties holds:
\begin{enumerate}
\item for every $(F, B) \in \mathcal{F}$, $e(F) \leq \delta n^2$,
\item for every $(F, B) \in \mathcal{F}$, $F$ is edge-disjoint from $B$,
\item for every $(F, B) \in \mathcal{F}$, $e(B) \geq \alpha n^2$,
\item for every graph $F$, there are at most $M$ pairs in $\mathcal{F}$ with first coordinate $F$. 
\end{enumerate}
\end{definition}
\noindent
For instance, $\mathcal{F}$ can consist of the pairs $(F, B)$ where $F$ is a sparse non-Hamiltonian expander and $B$ is the graph containing the boosters for $F$.

The following conditioning argument was introduced by Lee and Sudakov~\cite{LS12}, and essentially says that with high probability all sufficiently sparse expanders which appear as a subgraph in $G(n,p)$ also must have a booster in $G(n,p)$.
We will use a variation of their original statement, which appears as Lemma~4.5 in~\cite{APS2022}.
We note that the original lemma did not include a quantitative bound on the failure probability, but this bound follows from the last line of the proof in~\cite[Lemma 4.5]{APS2022}.

\begin{lemma}\label{lemma:boosterable}
For every $\alpha > 0$, there exists $\delta = \delta(\alpha) > 0$ with the following properties.
Fix $G\sim G(n,p)$ with $np \geq 1$ and let $\mathcal{F}$ be a $(\delta p, \alpha, 1)$-boosterable family.
Then, with probability at least $1-\e^{-\alpha p n^2/32}$, for every $(F, B) \in \mathcal{F}$, if $F \subseteq G$, then $E(G) \cap E(B) \neq \emptyset$.
\end{lemma}

Combining Lemma~\ref{lemma:boosterable} and a union bound, we immediately get the following corollary for general $(\delta, \alpha, M)$-boosterable families.

\begin{corollary} \label{corollary:boosterableM}
For every $\alpha > 0$, there exists $\delta = \delta(\alpha) > 0$ with the following properties.
Fix $G\sim G(n,p)$ with $np \geq 1$ and let $\mathcal{F}$ be a $(\delta p, \alpha, M)$-boosterable family.
Then, with probability at least $1-M\e^{-\alpha p n^2/32}$, for every $(F, B) \in \mathcal{F}$, if $F \subseteq G$, then $E(G) \cap E(B) \neq \emptyset$.
\end{corollary}

In our application, we let $\mathcal{F}$ consist of the pairs $(F, B)$ obtained as follows.
Given $\delta$ and $\lambda$, let $\mathcal{G}_{\delta, \lambda}$ be the family of graphs $G$ on $n$ vertices such that, first, the $2$-core $\Co(G)$ of $G$ has at least $n - n^{11/12}$ vertices and, second, at most $n^{11/12}$ vertices have degree less than $\delta \lambda$ in $G$.
Furthermore, we define $\mathcal{H}_{\delta, \lambda}^*$ as the family of reduced $(x_1,x_2)$-cores originating from graphs $G \in \mathcal{G}$ and vertices $x_1,x_2$ satisfying~\ref{item:x1x2farapart} and~\ref{item:x1x2farfromsmall}.
Then, the family $\cF$ consists of the following pairs: given $H^*\in \mathcal{H}_{\delta, \lambda}^*$ and a spanning connected $(2, n/7)$-expander $F^\ast \subseteq H^\ast$ with at most $2 \delta n \log n$ edges, the pair $(F,B)$ where
\begin{align*}
    F & = F^\ast[S^\ast \cup U^\ast]\text{, and } \\
    B & = \{\text{boosters for $F^\ast$ outside $F^*$ and with both endpoints in $U^\ast$}\}
\end{align*}
belongs to $\cF$. Note that the same graph $F$ could appear in multiple pairs in $\mathcal{F}$ for two reasons. 
First, for different spanning connected $(2, n/7)$-expanders $F_1^*, F_2^*$ having the same sets $S^*$ and $U^*$, the sets of boosters in $U^*$ for $F_1^*$ and $F_2^*$ may depend on the edges towards $T^*$.  
Second, even if the sets $S^*\cup U^*$ coincide for two spanning connected $(2, n/7)$-expanders, the sets $U^*$ may be different, which again could result in different sets of boosters.
Next, we show that, for all $F$, the number of pairs in $\mathcal{F}$ with first coordinate $F$ is suitably bounded.

\begin{lemma}\label{lemma:copiesosF}
Fix $\delta$ sufficiently small, $\lambda = np$ and $\cF$ as described above.
Then, $\mathcal{F}$ is a $(\nu, \alpha, M)$-boosterable family with $\nu = (2 \delta \log n) / n$, $\alpha = 1/200$ and $M = 2^{5 n^{11/12} \log n}$.
\end{lemma}
\begin{proof}
The first two conditions of Definition~\ref{def:bf} are satisfied by construction and the third one is satisfied by Lemma~\ref{lemma:therearemanyboosters}. It remains to verify the fourth condition.

Fix a graph $F$ serving as first coordinate of some element in $\cF$ and originating from restricting a spanning connected $(2, n/7)$-expander $F^*\subseteq H^*$ to $S^*\cup U^* = V(F)$.
To begin with, the number of possible choices for $S^*\subseteq V(F)$ of size $|S^*|\le n^{11/12}$ is at most
\[\sum_{i=0}^{n^{11/12}} \binom{n}{i}\le n \bigg(\frac{\e n}{n^{11/12}}\bigg)^{n^{11/12}}\le 2^{n^{11/12}\log n}.\]
Moreover, suppose that $H^*$ is a reduced core with fixed sets $S^*$ and $U^*$.
Then, the size of $T^*$ is at most $|S|/2+1\le n^{11/12}$ and, since exactly two edges are incident to every vertex in $T$, there are at most $(n^2)^{n^{11/12}}\le 2^{4n^{11/12}\log n}$ ways to extend a fixed graph on $S^*\cup U^*$ to a reduced core $H^*$.
As the pair of sets $S^*,U^*$ and the edges towards $T^*$ determine the boosters for $F^*$ in $U^*$, by multiplying the above upper bounds, we obtain that $M = 2^{5n^{11/12}\log n}$ is an upper bound on the number of pairs in $\cF$ containing $F$, as desired.
\end{proof}

\subsection{\texorpdfstring{Proof of Lemma \ref{lemma:hamcon2core} and Remark~\ref{rem:Luc}}{}}\label{sec:5.3}
We are now ready to give the proof of Lemma~\ref{lemma:hamcon2core}, the main aim of this section.
We also give the argument for Remark~\ref{rem:Luc} as a corollary afterwards.

\begin{proof}[Proof of Lemma \ref{lemma:hamcon2core}]
Fix $\eps > 0$ and $\delta = \delta(\eps) > 0$ sufficiently small, $np$ in the interval $[(1/3 + \eps) \log n, (1-\eps) \log n]$ and $D = \delta np$.
Recall that, by Lemma~\ref{lemma:components-small}, the following properties hold with probability $1-O(n^{- \eps})$:
\begin{enumerate}[\textbf{C\arabic*}]
\item \label{item:noclosesmallvertices*} for every set $R \subseteq V(G)$ with $|R|\le 100$ and such that the graph $G[R]$ is connected, $|R \cap S(D)| \leq 2$,
\item \label{item:nocyclesmallvertices*} for every set $R \subseteq V(G)$ with $|R|\le 100$ and such that $G[R]$ has at least as many edges as vertices, $R \cap S(D) = \emptyset$,
\item \label{item:nosmallcomponents*} $G$ contains no connected components of size $k\in [3,n/2]$.
\end{enumerate}
For $H = \Co(G)$ and $x_1, x_2 \in V(H)$, recall the properties
\begin{enumerate}[\textbf{C\arabic*}]
\setcounter{enumi}{3}
\item \label{item:x1x2farapart*} $\operatorname{dist}_H(x_1, x_2) \geq 8$,
\item \label{item:x1x2farfromsmall*} $x_1, x_2$ are $D$-far in $H$.
\end{enumerate}
Moreover, by Lemma~\ref{lemma:sparseandjoined}, with probability $1-o(1/(\log n)^2)$, we also have
\begin{enumerate}[\upshape{\textbf{C\arabic*}}]
\setcounter{enumi}{5}
\item \label{item:sparsification-ii*} every set $X \subseteq V(G)$ of size $|X| \leq n/(\log n)^{1/2}$ spans at most $(\log n)^{3/4}|X|$ edges in $G$,
\item \label{item:sparsification-iii*} for every pair of disjoint sets $X, Y\subseteq V(G)$ of sizes $|X| \leq n/(\log n)^{1/2}$ and $|Y| \leq (\log n)^{1/4} |X|$, the number of edges in $G$ between $X$ and $Y$ is at most $D|X|/2$, and
\item \label{item:suitable-7*} for every pair of disjoint sets $X, Y\subseteq V(G)$ of size $\lceil n/(\log n)^{1/2} \rceil$, the number of edges in $G$ between $X$ and $Y$ is at least $n/6$.
\end{enumerate}
Recall also that, by Lemma~\ref{lemma:n^11/12}, with probability $1-o(1/(\log n)^2)$,
\begin{enumerate}[\upshape{\textbf{C\arabic*}}]
\setcounter{enumi}{8}
\item \label{item:sparsification-proof-smallbound*} $|S_G(D)| \leq n^{11/12}$, and
\item \label{item:sparsification-proof-maxdeg*} $\Delta(G) \leq 12 \log n$.
\end{enumerate}
	
Finally, let $\mathcal{F}$ be the family of pairs $(F, B)$ described in the previous subsection.
Then, by Lemma~\ref{lemma:copiesosF}, $\mathcal{F}$ is a $(\nu, \alpha, M)$-boosterable family where $\nu = (2 \delta \log n) / n \leq 6 \delta p$, $\alpha = 1/200$ and $M = 2^{5n^{11/12} \log n}$.
Thus, $M \e^{-\alpha p n^2 / 32} = o(1/(\log n)^2)$ and, by Corollary~\ref{corollary:boosterableM}, with probability $1-o(1/(\log n)^2)$,
\begin{enumerate}[\upshape{\textbf{C\arabic*}}]
\setcounter{enumi}{10}
\item \label{item:thereareboosters*} for each $(F, B) \in \mathcal{F}$, if $F \subseteq G$, then $E(G) \cap B \neq \emptyset$.
\end{enumerate}
	
In summary, the properties \ref{item:noclosesmallvertices*}--\ref{item:thereareboosters*} hold jointly with probability $1-o(1/(\log n)^2)$.
We show the desired outcome conditionally on these properties.
Consider the reduced $(x_1,x_2)$-core $H^\ast = H^\ast(G, D, x_1, x_2)$.
Since \ref{item:noclosesmallvertices*}--\ref{item:nocyclesmallvertices*} and \ref{item:x1x2farapart*}--\ref{item:x1x2farfromsmall*} hold, $H^\ast$ is well-defined.
By Lemma~\ref{lemma:hstarhamilton}, $H$ contains a Hamilton $(x_1, x_2)$-path if $H^\ast$ is Hamiltonian, which~we~now~prove.

By \ref{item:sparsification-proof-smallbound*} and Lemma~\ref{lemma:degreeincore}(iii), we have that $|V(H)| \geq n - n^{11/12}$.
Hence, for any reduced graph $H^\ast$ with sets $S^\ast, T^\ast, U^\ast$ obtained from $H$ and any non-Hamiltonian spanning connected $(2, n/7)$-expander $F^\ast \subseteq H^\ast$ with at most $2 \delta n \log n$ edges,
the pair $(F,B)$ with $F = F^\ast[S^\ast \cup T^\ast]$ and $B$ containing all boosters for $F^*$ in $S^\ast \cup T^\ast$ is in $\cF$.
Then,~\ref{item:thereareboosters*} implies that there exists $B\cap G[U^\ast]\neq \emptyset$.
In the sequel, we will use this observation repeatedly.
	
We define a sequence of graphs $F_0 \subseteq F_1 \subseteq \dotsb \subseteq F_n \subseteq H^\ast$ such that, for each $i\in [0,n]$,
\begin{enumerate}
\item $V(F_i) = V(H^\ast)$,
\item $e(F_i) \leq \delta n \log n + i$,
\item for $i > 0$, either $F_i$ is Hamiltonian or its longest path is longer than that of $F_{i-1}$.
\end{enumerate}
Initially, we let $F_0 \subseteq H^\ast$ be a spanning connected $(2, n/7)$-expander which is also a $D$-sparsification of $H^\ast$ for $D = \delta \log n$ (which exists by Lemma~\ref{lemma:existenceexpander}).
In particular, $F_0$ has at most $\delta n \log n$ edges.
For every $i\in [n]$, if $F_{i-1}$ is Hamiltonian, set $F_{i-1} = F_i$.
Otherwise, $e(F_{i-1}) \leq \delta n \log n + i-1 \leq 2 \delta n \log n$.
Thus, by the definition of $\cF$ and \ref{item:thereareboosters*}, there exists a booster $e$ in $G[U^\ast] \subseteq H^\ast$.
Then, we define $F_i = F_{i-1} + e$, which satisfies the required properties.
Since the longest path in $H^\ast$ has at most $n$ edges, we have that $F_n$ (and hence, $H^\ast$) must be Hamiltonian, as required.
\end{proof}

\begin{proof}[Proof of Remark~\ref{rem:Luc}]
Fix $N = n(n-1)/2$ and a binomial random variable $X \sim \mathrm{Bin}(N,p)$. Then, consecutively sample $X$ distinct edges of $K_n$ and let $G_1$ be the random graph containing the first $\max(X-1,0)$ of them. Then, on the one hand, a simple counting argument shows that the total variation distance between $G_1$ and $G(n,p)$ is equal to
\begin{align*}
\frac{1}{2} \sum_{i=0}^{N} \binom{N}{i} \bigg|p^i (1-p)^{N-i} - \mathds{1}_{\{i<N\}} \frac{N-i}{i+1} p^{i+1} (1-p)^{N-i-1}\bigg|.
\end{align*}
Using that $\mathbb P(X\neq Np \pm (Np)^{2/3}) = o(1/(\log n)^{2})$ and that, for all $i = Np\pm (Np)^{2/3}$, we~have $(N-i)p = (i+1)(1-p) + O(ip+(Np)^{2/3}) = (1+o(1/(\log n)^{2})) (i+1)(1-p)$, the above expression further rewrites as
\begin{align*}
o((\log n)^{-2}) + \frac{1}{2} \sum_{i=\lfloor Np - (Np)^{2/3}\rfloor}^{\lceil Np + (Np)^{2/3}\rceil} \binom{N}{i} p^i (1-p)^{N-i} \bigg|1 - \frac{(N-i)p}{(i+1)(1-p)}\bigg| = o((\log n)^{-2}).
\end{align*}

On the other hand, for all sufficiently small $\delta > 0$, 
the number of vertices which are not $\delta np$-far is dominated by $|S_G(\delta np)| \Delta(G)^8$, and we know by~\ref{item:sparsification-proof-smallbound*} and~\ref{item:sparsification-proof-maxdeg*} that this expression is of order $O(n^{11/12} (\log n)^8)$ with probability $1-o(1/(\log n)^2)$.
Moreover, by \ref{item:sparsification-proof-smallbound*} and Lemma~\ref{lemma:degreeincore}(iii), the 2-core of $G(n,p)$ contains at least $n - n^{11/12}$ vertices with probability $1-o(1/(\log n)^2)$.
Conditionally on the previous two observations and the event from Lemma~\ref{lemma:hamcon2core}, for a uniformly chosen pair of vertices $x,y$ not forming an edge in $G(n,p)$, there is a Hamilton path of the 2-core of $G(n,p)$ connecting $x$ and $y$ with probability $1-o(1/(\log n)^2)$. 
Moreover, the bound on the total variation between $G_1$ and $G(n,p)$ yields that the same is valid for $G_1$. 
Thus, the additional edge completes the latter Hamilton path to a Hamilton cycle of the 2-core of $G_1$ with the required probability, as desired.
\end{proof}

\section{\texorpdfstring{Separating sparse random graphs: proof of Theorem~\ref{thm:main_ER}(iii)}{}}\label{sec:6}

The lower bound in Part~(iii) of Theorem~\ref{thm:main_ER} follows by combining Lemma~\ref{lem:X1} and the fact that, when $np\in [1,(1-\eps)\log n]$, the number of leaves in $G$ is concentrated around its mean (see Lemma~\ref{lem:C}).
In this section, we show the upper bound.

\begin{proposition}\label{prop:main-sparse}
For all $\delta,\eps > 0$, there exists $C = C(\delta)$ such that, for $G\sim G(n,p)$ with $np\in [C, (1-\eps)\log n]$, a.a.s.
\begin{equation}\label{eq:fst ord}
\sp(G) \leq (2/3 + \delta) n^2 p \e^{-np}.  
\end{equation}
\end{proposition}

We note that our approach is sufficiently flexible to provide an improved bound on the second order term in~\eqref{eq:fst ord} conditionally on the number of leaves and isolated vertices in $G$.
	
\subsection{Preliminaries}

First, we prepare the ground by announcing a few preliminary results.
The following lemma appears as Corollary~25 in~\cite{Mon18}; its proof is based on the same conditioning argument due to Lee and Sudakov~\cite{LS12} used in the previous section.
	
\begin{lemma}[Corollary~25 in~\cite{Mon18}]\label{lem:LS}
Fix $G\sim G(n,p)$ with $np\geq 10^5$. Then, a.a.s.\ every $(2,n/5)$-expander $H\subseteq G$ with $e(H) \leq 5n^2p/10^5$ satisfies that $G[V(H)]$ is Hamiltonian.
\end{lemma}

Following Lemma~\ref{lem:LS}, we will look for expanders with suitable vertex sets in $G(n,p)$. 
The next lemmas show that certain induced subgraphs of $G(n,p)$ have good expansion properties.
They appear as Proposition~13 and Proposition~20 in~\cite{Mon18}. 

\begin{lemma}[Proposition~13 in~\cite{Mon18}] \label{lem:Mon13}
Fix $G\sim G(n,p)$ with $np$ sufficiently large. Then, a.a.s.\ every set $A\subseteq V(G)$ with $|A| = \lceil 100/p\rceil$ satisfies $|N_G(A)| \geq 9n/10$.
\end{lemma}
	
\begin{lemma}[see Proposition~20 in~\cite{Mon18}]\label{lem:Mon20b}
Fix $G\sim G(n,p)$ with $np\le 2\log n$ sufficiently large.
Then, a.a.s.\ no set $A \subseteq V(G)$ with $|A|\leq n/10^{15}$ satisfies $e(G[A])\geq np|A|/10^7$.
\end{lemma}

Next, we provide an easy estimate on the number of short cycles in random graphs.

\begin{lemma} \label{lemma:easy-cycles}
Fix $G\sim G(n,p)$ with $np\le 2\log n$.
Then, a.a.s.\ there are at most $(\log n)^5$ cycles of length $3$ or $4$ in $G$.
\end{lemma}
\begin{proof}
The expected number of $3$-cycles in $G$ is at most $n^3 p^3 \leq 8 (\log n)^3$ and the expected number of $4$-cycles in $G$ is at most $n^4 p^4 \leq 16 (\log n)^4$.
The statement follows by Markov's inequality for the sum of these two random variables.
\end{proof}
	
The next lemma quantifies the proportion of leaves which lie outside the giant component of a random graph away from criticality. It follows from a combination of our Lemma~\ref{lem:C} and Exercise~3.3.2 in~\cite{FK16}.
	
\begin{lemma}[see Exercise~3.3.2 in~\cite{FK16}] \label{lemma:fewleavesoutsidegiant}
Fix $\delta,\eps\in (0,1)$ and $G\sim G(n,p)$ for sufficiently large $np \le (1-\eps)\log n$.
Then, a.a.s.\ there are at most $\delta n^2p \e^{-np}/20$ leaves in $G$ outside the giant component of $G$.
\end{lemma}
	
The next lemma estimates the number of vertices of small degree in a random graph (seen as trivial paths of length 0) and the number of short paths between those. It follows from a combination of our Lemma~\ref{lem:C} and Proposition~19 in~\cite{Mon18}.
Recall that, given a graph $G$ and a set $S\subseteq V(G)$, an $S$-path is a path of length at most 4 whose both endpoints are in $S$, and an $S$-cycle is a cycle of length at most 4 containing a vertex in $S$. 
	
\begin{lemma}[see Proposition~19 in~\cite{Mon18}]\label{lemma:fewSpaths}
Fix $\eps\in (0,1)$ and $G\sim G(n,p)$ for sufficiently large $np\le (1-\eps)\log n$, and set $S = \{v : \deg_G(v)\leq np / 100\}$. Then, a.a.s.\ $|S|\le \e^{-np/2} n$ and the number of $S$-paths and $S$-cycles in $G$ is at most $\e^{- 3np/2} n$.
\end{lemma}

Finally, our proof makes use of the following theorem of Hajnal and Szemer\'edi~\cite{HS70}. 
Given a graph $G$ and $k$ colours, we call a vertex-colouring \emph{equitable} if every vertex gets a single colour, no neighbours in $G$ share the same colour and the sizes of every pair of colour classes differs by at most 1.

\begin{theorem}[Hajnal-Szemer\'edi theorem, see~\cite{HS70}]\label{thm:HS}
Every graph with maximum degree $\Delta$ has an equitable colouring in $\Delta+1$ colours.
\end{theorem}
	
\subsection{Proof of Proposition~\ref{prop:main-sparse}}
The proof of Proposition~\ref{prop:main-sparse} is divided into four steps described roughly as follows. 
Define $S = S_G(np/100)$.
First, we find a sparse subgraph of $G$ where Lemma~\ref{lem:LS} can be applied (here we use techniques from~\cite[Section 4]{Mon18}). Next, we construct a path system which covers and separates all isolated vertices, leaves and vertices lying on $S$-paths and $S$-cycles.
Then, we construct another path system which covers and separates the vertices in $S$ that do not belong to the previous set.
Finally, we construct a path system that separates the vertices outside $S$.
The path system required~in Proposition~\ref{prop:main-sparse} is obtained as a union of the path systems constructed in the last three~steps.
	
\begin{proof}[Proof of Proposition~\ref{prop:main-sparse}]
Fix $\delta,\eps > 0$, and assume that $np$ is sufficiently large for each of the following computations to hold. Define $S_0, S_1$ as the sets of isolated vertices and leaves in $G$, respectively. By Lemmas~\ref{lem:max_deg} and~\ref{lem:LS}--\ref{lemma:fewSpaths}, a.a.s.\ the following properties hold.
\begin{enumerate}[\upshape{\textbf{E\arabic*}}]
\item \label{item:sparse-hamiltonian} for any $(2,n/5)$-expander $H \subseteq G$ with $e(H) \leq 5 n^2 p/10^5$, $G[V(H)]$ is Hamiltonian,
\item \label{item:sparse-densespots} for every set $A \subseteq V(G)$ with $|A| \leq n/10^{15}$, we have $e(G[A]) < np |A|/10^7$,
\item \label{item:sparse-maxdegree} $\Delta(G) \leq 12 \log n$,
\item \label{item:sparse-cycles} there are at most $(\log n)^5$ cycles of length $3$ or $4$ in $G$,
\item \label{item:sparse-giantleaves} there are at most $\delta n^2 p \e^{-np}/20$ leaves outside the giant component of $G$,
\item \label{item:sparse-spaths} there are at most $\e^{-3np/2} n$ $S$-paths and $S$-cycles in $G$, and 
\item \label{item:sparse-ssize} $|S| \leq \e^{- np/2} n$.
\end{enumerate}
Moreover, Lemma~\ref{lem:C} shows that a.a.s.\
\begin{enumerate}[\upshape{\textbf{E\arabic*}}, resume]
\item \label{item:sparse-deg0} $|S_0| = (1+o(1))n \e^{-np}$, and
\item \label{item:sparse-deg1} $|S_1| = (1+o(1))n^2 p \e^{-np}$.
\end{enumerate}
In the sequel, we assume that $G$ satisfies the properties \ref{item:sparse-hamiltonian}--\ref{item:sparse-deg1}. \medskip

\noindent \emph{Step 1: Passing to a sparser subgraph.} For every vertex in the graph $G$, mark arbitrarily exactly $\min(\deg_G(v), np / (2\cdot 10^5))$ edges incident to that vertex (one edge could be marked twice for each of its endpoints).
Let $G_0$ be the spanning subgraph of $G$ whose edge set coincides with the marked edges. Clearly, $e(G_0) \leq n^2 p / 10^5$.
Also, let $G_1 \subseteq G$ be a random subgraph of $G$ obtained by keeping each edge of $G$ with probability $1/10^5$ (so $G_1\sim G(n, p/10^5)$). 
Then, a.a.s.\ $e(G_1)\leq n^2 p/10^5$ and Lemma~\ref{lem:Mon13} (applied here with $G_1$ instead of $G$, and $p/10^5$ in place of $p$) shows that every $A \subseteq V(G)$ with $|A| = \lceil 10^7 / p \rceil$ satisfies $|N_{G_1}(A)| \geq 9n/10$.
		
Let $G' = G_0 \cup G_1$. By the above properties, a.a.s.
\begin{enumerate}[\upshape{\textbf{E\arabic*}}, resume]
\item \label{item:sparse-sparser} $e(G') \leq 2 n^2 p / 10^5$,
\item for every $v \notin S_0 \cup S_1$, $\deg_{G'}(v) \geq 2$,
\item for every $v \notin S$, $\deg_{G'}(v) \geq np / (2\cdot 10^5)$, and
\item \label{item:sparse-joined} every $A \subseteq V(G)$ with $|A| = \lceil 10^7/p\rceil$ satisfies $|N_{G'}(A)| \geq 9n/10$.
\end{enumerate}	
Moreover, $S_0$ and $S_1$ are the sets of isolated vertices and leaves of $G'$, respectively, and the number of $S$-paths and $S$-cycles in $G'$ is at most 
\begin{equation}\label{eq:Y}
Y = \e^{-3np/2}n\le \delta n^2p \e^{-np} /100.
\end{equation}

Define $B_0$ to be the union of $S_0\cup S_1$ and the set of vertices belonging to at least one $S$-path or $S$-cycle in $G'$. In particular, $|B_0| \leq |S_0|+|S_1|+5 Y$ and $|B_0 \setminus S| \leq 3 Y$ but note that $S$ is not necessarily included in $B_0$.
Moreover, define the subset $B_1$ of $V(G)\setminus B_0$ as follows: if there is a set $W$ such that $|W|\le Y$ and $e(G'[(B_0 \cup W) \setminus S]) \geq np |W|/(2 \cdot 10^6)$, define $B_1$ as a largest set with this property, and otherwise, define $B_1 = \emptyset$.
We will show that $|B_1| < Y$; roughly speaking, this means that, for every set $W$ of size $Y$ disjoint from $B_0$, $(B_0\cup W)\setminus S$ does not span too many edges in $G'$.
Setting $\bar B_1 = (B_0 \cup B_1) \setminus S$, we have 
\[|\bar B_1|\le |B_0\setminus S| + |B_1|\le 4Y\le n/10^{15}.\] 
In particular, this implies that $|B_1| < Y$ since otherwise $|\bar B_1|\le 4|B_1|$ and~\ref{item:sparse-densespots} implies that
\[e(G'[\bar B_1])\le e(G[\bar B_1])\le np |\bar B_1|/10^7 < np |B_1|/(2\cdot 10^6),\]
thus contradicting the definition of $B_1$.
Define $B_2 = N_{G'}(B_1) \cap S$.
Since $B_0 \cap B_1 = \emptyset$, each vertex $v\in B_1$ can have at most one neighbour in $S$ as otherwise $v$ would lie on an $S$-path of length two.
Hence, $|B_2| \leq |B_1| < Y$. We finish the construction of our desired sparse subgraph of $G'$ by setting $H = G'\setminus (B_0 \cup B_1 \cup B_2)$.
In order to prove that some appropriate subgraphs of $H$ are expanders, we show the following four properties of $H$:
		
\begin{enumerate}[\upshape{\textbf{E\arabic*}}, resume]
\item \label{item:sparse-neighboursSS} for every $u,v\in S\cap V(H)$, $u$ and $v$ are at distance at least 5 in $H$,
\item \label{item:sparse-neighboursHS} for every $w \in V(H)$, $w$ has at most one neighbour in $S\cup N_{G'}(S)$,
\item \label{item:sparse-largevertices} for every $v \in V(H) \setminus S$, $|N_H(v) \setminus S| \geq 4 np /10^6$, and
\item \label{item:sparse-mindegree} $\delta(H)\ge 2$.
\end{enumerate}			

\begin{claim}\label{cl:E}
The properties \ref{item:sparse-neighboursSS}--\ref{item:sparse-mindegree} are all satisfied.
\end{claim}
\begin{proof}[Proof of Claim~\ref{cl:E}]
First,~\ref{item:sparse-neighboursSS} and~\ref{item:sparse-neighboursHS} follow from the fact that $H$ contains no $S$-paths and $S$-cycles.
To prove~\ref{item:sparse-largevertices}, note that, since $v\notin B_0\cup S$, $\deg_{G'}(v) \geq np / (2\cdot 10^5)$ and $v$ has at most one neighbour in $S$ (and therefore, at most one neighbour in $B_2$). Moreover, if $v$ has at least $np / (2\cdot 10^6) + 1$ neighbours in $B_0 \cup B_1$, then 
\begin{align*}
e(G[(B_0 \cup B_1 \cup \{v\}) \setminus S]) 
&\geq e(G[(B_0\cup B_1)\setminus S]) + |N_{G'}(v)\cap (B_0\cup B_1)| - 1\\
&\geq \lambda (|B_1|+1)/(2\cdot 10^6), 
\end{align*}
which contradicts the maximality of $B_1$. Hence,
\[|N_H(v) \setminus S|\geq np /(2\cdot 10^5) - 1 - np / (2 \cdot 10^6) - 1 \geq 4 np /10^6.\]

We turn to~\ref{item:sparse-mindegree}. Fix a vertex $v$ in $H$. If $v$ does not belong to $S$, then~\ref{item:sparse-largevertices} already implies that $v$ has at least two neighbours in $H$. Suppose that $v\in S\cap V(H)\subseteq S\setminus (S_0\cup S_1)$ and some of its neighbours belongs to $B_0\cup B_1\cup B_2$.
If $v$ has a neighbour in $B_0$, then either this neighbour is in $S$ and $v$ belongs to an $S$-path of length 1 (so to $B_0$ as well), 
or this neighbour lies on an $S$-path or an $S$-cycle, which means that $v$ is an endpoint of an $S$-path of length at most $4/2+1 = 3$.
If $v$ has a neighbour in $B_1$, then it would belong to $B_2$, which is disjoint from $V(H)$.
Finally, if $v$ has a neighbour in $B_2\subseteq S$, then it lies on an $S$-path of length 1.
Thus, our assumption that $v\in S\setminus (S_0\cup S_1)$ has a neighbour in $B_0\cup B_1\cup B_2$ leads to a contradiction in every case, so $\deg_H(v) = \deg_{G'}(v)\ge 2$, which proves~\ref{item:sparse-mindegree}.
\end{proof}

\noindent \emph{Step 2: Separating the leaves.}
Define $\mathcal{F}_0$ to be a path system containing $(B_0 \setminus S_1) \cup B_1 \cup B_2$ as trivial paths of length 0.
Recalling that $|B_0 \setminus S_1| \leq |S_0|+5Y$, $|B_2|\le |B_1|\le Y$, \eqref{eq:Y} and~\ref{item:sparse-deg0}, we get that
\[|\mathcal{F}_0| \leq |S_0| + 7Y \leq \delta n^2p \e^{-np}/10.\]
Note that $\mathcal{F}_0$ covers and separates all vertices of $V(G) \setminus V(H)$ except those of $S_1$.
		
For the vertices in $S_1$, we construct a path system $\mathcal{F}_1$ as follows. First, add every leaf outside the giant component to $\cF_1$ as a path of length 0. 
Then, find a maximal disjoint collection of triplets $\{y_1, y_2, y_3\}$ of distinct leaves in the giant component, and add to $\cF_1$ one path from $y_1$ to $y_2$ and one path from $y_2$ to $y_3$ for every such triplet. 
Finally, if some leaves in the giant component belong to no triplet as above, add these as paths of length 0 to $\cF_1$.
Then, $\mathcal{F}_1$ covers and separates $S_1$, and~\ref{item:sparse-giantleaves} and~\ref{item:sparse-deg1} show that
\[ |\mathcal{F}_1| \leq \delta n^2p \e^{-np}/20 + 2|S_1|/3 + 2 \leq \left(2/3+\delta/10\right) n^2p \e^{-np}.\]
We conclude that the system $\mathcal{F}_0 \cup \mathcal{F}_1$ both covers and separates $B_0\cup B_1\cup B_2$ and contains at most $(2/3 + \delta/5)n^2p \e^{-np}$ paths. \medskip
		
\noindent \emph{Step 3: Separating the vertices of small degree.}
It remains to separate the vertices in $V(H)$. We begin by constructing a path system that covers and separates $V(H) \cap S$. 

\begin{claim}\label{cl:E18}
The following property is satisfied:
\begin{enumerate}[\upshape{\textbf{E\arabic*}}, resume]
\item \label{item:sparse-pathinS} for every $T \subseteq S \cap V(H)$, there is a path $P_T$ such that $V(P_T) = (V(H) \setminus S) \cup T$.
\end{enumerate}
\end{claim}

\begin{proof}[Proof of Claim~\ref{cl:E18}]
Fix $T \subseteq S \cap V(H)$ and denote $V_T = (V(H) \setminus S) \cup T$. We show that $H[V_T]$ is a $(2,n/5)$-expander, which is enough to conclude that $H[V_T]$ is Hamiltonian by combining~\ref{item:sparse-hamiltonian},~\ref{item:sparse-sparser} and the fact that $H\subseteq G'$ (so $e(H)\le e(G')\le 2 n^2p/10^5$).
Let $A \subseteq V_T$ be a set of size $|A|\leq n/5$. We show that $|N_{H}(A)\cap V_T| \geq 2 |A|$ by considering two cases. 
First, suppose that $|A| \geq n/10^{16}\ge 10^7/p$ and let $A'\subseteq A$ be an arbitrary set of size $10^7/p$. Then,~\ref{item:sparse-joined} implies that 
\begin{equation}
\begin{split}\label{eq:A17}
|N_H(A)\cap V_T|
&= |N_H(A')\cap V_T| - |A\setminus A'|\\
&\ge |N_{G'}(A')| - |B_0\cup B_1\cup B_2| - |S| - |A|\\
&\ge 9n/10 - |S_0|-|S_1|-7Y-|S| - |A|\ge 2|A|.
\end{split}
\end{equation}
It remains to consider the case $|A| < n/10^{16}$. Define $A_1 = A \cap S$ and $A_2 = A \setminus S$.
Thanks to~\ref{item:sparse-neighboursSS}, we know that $A_1$ is an independent set and its vertices have no common neighbours in $H$. Moreover, by~\ref{item:sparse-mindegree}, we have that $\delta(H)\geq 2$,  so $|N_H(A_1) \cap V_T| \geq 2 |A_1|$.
We turn to $A_2$. Set $Z = N_H(A_2)\setminus S$ and suppose that $|Z| \leq 3 |A_2|$. Then, $|Z \cup A_2| \leq 4 |A_2| \leq n/10^{15}$ and \ref{item:sparse-largevertices} implies that $Z \cup A_2$ spans at least $2 np |A_2|/10^6 \geq np |Z \cup A_2|/10^7$ edges of $H$, which contradicts \ref{item:sparse-densespots}.
This proves that $|Z| \ge 3 |A_2|$. Moreover, using the fact that, by \ref{item:sparse-neighboursHS}, every vertex in $V(H)\setminus S\supseteq A_2$ has at most one neighbour in $N_H(S)\supseteq N_H(A_1)$, we have that $|Z \cap N_H(A_1)| \leq |A_2|$. Hence,
\[|N_H(A) \cap V_T| \geq |N_H(A_1)| + |Z| - |Z \cap N_H(A_1)| \geq 2 |A_1| + 3 |A_2| - |A_2| \geq 2 |A|,\]
which proves that $H[V_T]$ is indeed a $(2,n/5)$-expander, thus implying~\ref{item:sparse-pathinS}.
\end{proof}

Finally, by~\ref{item:sparse-pathinS}, there is a path system $\cF_S$ containing at most $\lceil \log_2 |S| \rceil + 1\le \log_2 n$ paths in $H$ that simultaneously covers and separates the set $S\cap V(H)$. 
Thus, the only vertices not separated by the path system $\mathcal{F}_0 \cup \mathcal{F}_1 \cup \mathcal{F}_S$ remain those in $V(H) \setminus S$. \medskip

\noindent \emph{Step 4: Separating the vertices of large degree.}
Consider the graph $H^2$ where two vertices are joined by an edge if they are at distance at most $2$ in $H$. By \ref{item:sparse-maxdegree}, we have $\Delta(H^2)\le (12 \log n)(12\log n-1)\le 144 (\log n)^2-1$.
In particular, the Hajnal-Szemer\'edi theorem implies that we can properly colour the vertices of $H \setminus S$ using $t \leq 144 (\log n)^2$ colours so that each colour class contains at most $\lceil n/t\rceil$ vertices.
Let $U_1, U_2, \dotsc, U_t$ be those colour classes. By construction, every pair of vertices in the same class $U_i$ are at distance at least $3$ in $H$.
		
The next property is analogous to \ref{item:sparse-pathinS} (and its proof is similar but slightly simpler).
		
\begin{enumerate}[\upshape{\textbf{E19}}]
\item \label{item:sparse-pathnotin S} For every $i\in [t]$ and every $T\subseteq U_i$, there is a path $P_T\subseteq H\setminus S$ with vertex set $V_T = (V(H) \setminus (S \cup U_i)) \cup T$.
\end{enumerate}

\begin{claim}\label{cl:E19}
Property~\ref{item:sparse-pathnotin S} is satisfied.
\end{claim}
\begin{proof}[Proof of Claim~\ref{cl:E19}]
Again, it is enough to show that $H[V_T]$ is a $(2,n/5)$-expander. Let $A \subseteq V_T$ be a set of size $|A|\leq n/5$. The case $|A| \geq n/10^{16}\ge 10^7/p$ is treated in the same way as in the proof of~\ref{item:sparse-pathinS} (except that $|S|$ has to be replaced by $|S\cup U_i|$ in~\eqref{eq:A17}), so we assume that $|A| < n/10^{16}$.
Since $V_T \cap S = \emptyset$, we have that $A \subseteq V(H) \setminus S$. Then, by defining $Z = N_H(A) \setminus S$, the same argument used to estimate $|Z|$ before shows that $|Z| \geq 3 |A|$.
Finally, we have that each vertex of $A$ has at most one neighbour in $U_i$ (otherwise, there would be two vertices in $U_i$ at distance less than $3$ in $H$), so $|N_H(A) \cap U_i| \leq |A|$. Hence, \[|N_H(A) \cap V_T| \geq |N_H(A) \setminus S| - |N_H(A) \cap U_i| \geq 3 |A| - |A| \geq 2 |A|,\] 
which proves that $H[V_T]$ is indeed a $(2,n/5)$-expander, thus implying~\ref{item:sparse-pathnotin S}.
\end{proof}

Then, analogously to the construction of $\mathcal{F}_S$, for every $i\in [t]$, one can construct a path system $\mathcal{F}_{U_i}$ of size at most $\lceil \log_2 n\rceil + 1$ that covers and separates the vertices of $U_i$ from each other and from the vertices in $V(H) \setminus (S\cup U_i)$. Finally, the path system $\mathcal{F} = \mathcal{F}_0 \cup \mathcal{F}_1 \cup \mathcal{F}_S \cup (\bigcup_{i=1}^t \mathcal{F}_{U_i})$ separates $V(G)$ by construction and has size at most
\begin{align*}
|\cF_0|+|\cF_1|+|\cF_S|+t\max_{i\in [t]} |\cF_i|
&\le (2/3+\delta/5) np \e^{-np} n + \log_2 n + (t+1)(\log_2 n+1)\\
&\le (2/3+\delta)np \e^{-np}n,
\end{align*}
which finishes the proof.
\end{proof}

\section{\texorpdfstring{Separating deterministic graphs: proof of Propositions~\ref{prop:example} and~\ref{prop:f(n)}}{}}\label{sec:7}

We begin this section with a proof of Proposition~\ref{prop:example}.

\begin{proof}[Proof of Proposition~\ref{prop:example}]
Set $k = \lfloor (n-6)/2\rfloor$, $\ell = n-6-k$ and consider the graph $G$ with vertices 
\[\{u_1,v_1,w_1,u_2,v_2,w_2\}\cup \{x_i:i\in [k]\}\cup \{y_i:i\in [\ell]\}\]
and edges 
\[\{u_1v_1,u_1w_1,u_2v_2,u_2w_2\}\cup \{v_1x_i,w_1x_i:i\in [k]\}\cup \{v_2y_i,w_2y_i:i\in [\ell]\}.\]
It is easy to check that $e(G) = 2n-8$.
First, let us show that $\sp(G)\ge (2k+2\ell)/3$. Indeed, fix a vertex-separating path system $\cF$ of $G$ and denote by $k_0$ (resp. $\ell_0$) the number of vertices $x\in \{x_i:i\in [k]\}$ (resp. $y\in \{y_i:i\in [\ell]\}$) for which there is a path $P\in \cF$ such that $V(P)\cap \{x_i:i\in [k]\} = \{x\}$ (resp. $V(P)\cap \{y_i:i\in [\ell]\} = \{y\}$).
Every path in $G$ contains at most three vertices in $\{x_i:i\in [k]\}\cup \{y_i:i\in [\ell]\}$ and, outside the $k_0+\ell_0$ vertices discussed above, all but at most one of the remaining vertices in $\{x_i:i\in [k]\}\cup \{y_i:i\in [\ell]\}$ must be covered by at least two paths.
Hence,
\[|\cF|\ge k_0+\ell_0+\left\lceil\frac{2(k+\ell-k_0-\ell_0-1)}{3}\right\rceil\ge \left\lceil\frac{2(k+\ell-1)}{3}\right\rceil = \left\lceil\frac{2(n-7)}{3}\right\rceil.\]

Now, we set $e=u_1u_2$ and construct a vertex-separating path system $\cF_e$ of $G\cup \{e\}$ containing at most $k+7$ paths. To do so, we first include each of $u_1,u_2,v_1,v_2,w_1,w_2$ as trivial paths of length 0 in $\cF_e$ and, if $\ell=k+1$, add $y_{\ell}$ to $\cF_e$ as well.
On top of these trivial paths, fix an integer $q\in [2,k-2]$ that is coprime with $k$ (this exists for every $k\ge 7$, which is ensured by our assumption that $n\ge 60$) and, for every $i\in [k]$, add the path $x_i v_1 x_{i+1} w_1 u_1 u_2 w_2 y_{qi} v_2 y_{qi+1}$ to $\cF_e$ where indices are seen modulo $k$.
On the one hand, the above paths separate the sets $\{x_i:i\in [k]\}$ and $\{y_i:i\in [k]\}$ by construction: while this is immediate for $\{x_i:i\in [k]\}$, it follows from the fact that there are no $i,j\in [k]$ such that $qi \equiv qj+1 \bmod k$ and $qj \equiv qi+1 \bmod k$ simultaneously.
On the other hand, for every $i\in [k]$, the two paths containing $x_i$ also contain $\{y_{qi}, y_{qi+1}\}$ and $\{y_{q(i-1)}, y_{q(i-1)+1}\}$ and one can check that these four vertices are all different by our choice of $q$.
Hence, this path system separates the set $\{x_i:i\in [k]\}$ from the set $\{y_i:i\in [k]\}$, so $\cF_e$ is a vertex-separating path system of $G$.
Finally,
\[\sp(G) - \sp(G\cup \{e\})\ge \left\lceil\frac{2(n-7)}{3}\right\rceil - (k+7) \ge \frac{2n-14}{3} - \frac{n-6+14}{2}\ge \frac{n}{6} - 10,\]
as desired.
\end{proof}

We turn to the proof of Proposition~\ref{prop:f(n)}.
A key ingredient in it is the celebrated theorem of P\'osa giving a sufficient condition for a graph to be Hamiltonian.

\begin{theorem}[see~\cite{Pos62}]\label{thm:Pos}
Let $G$ be a graph with degree sequence $d_1\le d_2\le \dotsb \le d_n$. Suppose that, for every integer $i\in [1,(n-2)/2]$, $d_i\ge i+1$ and, if $n$ is odd, $d_{\lceil n/2\rceil}\ge \lceil n/2\rceil$. Then, $G$ is Hamiltonian.
\end{theorem}



We are now ready to prove Proposition~\ref{prop:f(n)}.

\begin{proof}[Proof of Proposition~\ref{prop:f(n)}]
Fix $\ell = \lceil \log_2 n\rceil$, an $n$-vertex graph $G$ with minimum degree at least $n/2+9\sqrt{n\log\log n}$ and an injective assignment of vectors in $\{0,1\}^{\ell}$ to the vertices of $G$ uniformly at random. 
Also, for every $j\in [\ell]$, denote by $S_j$ the set of vertices $v$ such that the $j$-th coordinate of the vector associated to $v$ is 1.
Note that, for every $j\in [\ell]$, the random variables $(\mathds{1}_{w\in S_j})_{w\in V(G)}$ are negatively correlated since, for every set $U\subseteq V(G)$,
\begin{equation*}
\mathbb P\bigg(\bigcap_{w\in U} \{w\in S_j\}\bigg) = \prod_{i=0}^{|U|-1} \bigg(\frac{2^{\ell-1}-i+1}{2^{\ell}-i+1}\bigg) \le 2^{-|U|} = \prod_{w\in U} \mathbb P(w\in S_j).
\end{equation*}
Thus, Lemma~\ref{lem:PS} applied with $(\mathds{1}_{w\in S_j})_{w\in V(G)}$ and $t_0 = 2\sqrt{n\log\log n}$ implies that
\begin{equation}\label{eq:nc}
\mathbb P(|S_j| - n/2 \ge t_0)\le \exp\bigg(-\frac{t_0^2}{2(n/2+t_0/3)}\bigg) = o(1/\ell).
\end{equation}

Now, fix $j\in [\ell]$, a vertex $v$ in $G$ and condition on the event $\{v\in S_j\}$.
Then, the random variables $(1-\mathds{1}_{w\in S_j})_{w\in N(v)}$ are negatively correlated since, for every set $U\subseteq N(v)$,
\[\mathbb P\bigg(\bigcap_{w\in U} \{w\notin S_j\}\,\bigg{|}\, v\in S_j\bigg) = \prod_{i=0}^{|U|-1} \bigg(\frac{2^{\ell-1}-i}{2^{\ell}-1-i}\bigg) \le \bigg(\frac{2^{\ell-1}}{2^{\ell}-1}\bigg)^{|U|} = \prod_{w\in U} \mathbb P(w\notin S_j\mid v\in S_j).\]
Thus, for $X_{v,j} = |S_j\cap N(v)|$ with mean $\mu_{v,j} = (2^{\ell-1}-1)\deg(v)/(2^{\ell}-1)\ge n/4+2t_0$, 
we~get
\begin{equation}
\begin{split}\label{eq:X_v,j}
\mathbb P(X_{v,j}\le |S_j|/2)
&\le \mathbb P(X_{v,j}\le n/4+t_0) + \mathbb P(|S_j|\ge n/2+2t_0)\\
&\le \mathbb P(X_{v,j}\le \mu_{v,j}-t_0) + o(1/\ell)\\
&= \mathbb P((\deg(v) - X_{v,j}) - (\deg(v)-\mu_{v,j})\ge t_0) + o(1/\ell)\\
&\le \exp(-t_0^2/2(n+t_0/3)) + o(1/\ell) = o(1/\ell), 
\end{split}
\end{equation}
where the second inequality follows by combining the fact that $\mu_{v,j}\ge n/4+2t_0$ with~\eqref{eq:nc}, and the last inequality follows from Lemma~\ref{lem:PS} for the random variable $\deg(v) - X_{v,j}$.
In particular, for every $j\in [\ell]$, Markov's inequality for the number of vertices $v\in S_j$ satisfying $X_{v,j}\le |S_j|/2$ shows that, with probability $1-o(1/\ell)$, at least $8|S_j|/9$ vertices in $S_j$ have degree at least $|S_j|/2$ in $G[S_j]$.

Finally, a computation similar to~\eqref{eq:X_v,j} shows that 
\[\mathbb P(X_{v,j}\le 2|S_j|/5)\le \mathbb P(X_{v,j}\le 0.24n) + \mathbb P(|S_j|\ge 0.6n) = o(1/n\ell).\]
As a result, a union bound over the events 
\[\{\text{no more than $8|S_j|/9$ vertices in $S_j$ have degree at least $|S_j|/2$ in $G[S_j]$}\}\]
for all $j\in [\ell]$,
and over the events $\{X_{v,j}\le 2|S_j|/5\}$ for all $j\in [\ell]$ and $v\in S_j$, shows~that~a.a.s.
\begin{align*}
&\text{for all }j\in [\ell], G[S_j] \text{ contains at least } 8|S_j|/9 \text{ vertices of degree at}\\
&\text{least } |S_j|/2,\text{ and all vertices in } G[S_j] \text{ have degree at least } 2|S_j|/5.
\end{align*}
Under this event, we finish by applying Theorem~\ref{thm:Pos} to the graphs $G[S_1],\ldots,G[S_\ell]$.
\end{proof}

\subsection*{Acknowledgements.} Lichev was supported by the Austrian Science Fund (FWF) grant No. 10.55776/ESP624.
Sanhueza-Matamala was supported by ANID-FONDECYT Iniciaci\'on Nº11220269 grant.

\sloppy\printbibliography

\end{document}